\documentclass{article}

\usepackage[utf8]{inputenc}
\usepackage[british]{babel}
\usepackage{graphicx} 
\usepackage{graphicx}
\usepackage{amsmath}
\usepackage{amssymb}
\usepackage{amsthm}
\usepackage{multicol}
\usepackage{multirow}
\usepackage{caption}
\usepackage{subcaption}
\usepackage{wrapfig}
\usepackage{url}
\usepackage{array}
\usepackage[mathscr]{euscript}
\usepackage{float}
\usepackage{xspace}
\usepackage{tikz}
\usepackage{tikz-dependency}
\usepackage{tabularx}
\usetikzlibrary{shapes, arrows, calc, arrows.meta, fit, positioning} 
\tikzset{  
    -Latex,auto,node distance =1 cm and 0.7 cm, thick,
    state/.style ={circle, fill=lightgray, draw, minimum width = 0.01 cm}, 
    emptystate/.style ={}, 
    breakarrow/.style={dashed}
    point/.style = {circle, draw, inner sep=0.18cm, fill, node contents={}},  
    bidirected/.style={Latex-Latex,dashed}, 
    el/.style = {inner sep=2.5pt, align=right, sloped}  
} 
\renewcommand{\arraystretch}{1.5}
\usepackage{thmtools}
\usepackage{thm-restate}

\usepackage{hyperref}		
\hypersetup{
  colorlinks,
  citecolor=black,
  linkcolor=black,
  urlcolor=black
}

\newcommand{\4}{{\mathfrak{4}}\xspace}
\newcommand{\K}{\ensuremath{\mathcal{\mathbf{K}}}\xspace}
\newcommand{\RC}{\ensuremath{\mathcal{\mathbf{RC}}}\xspace}
\newcommand{\GLP}{\ensuremath{\mathcal{\mathbf{GLP}}}\xspace}
\newcommand{\TRC}{\ensuremath{{\sf TRC}}\xspace}
\newcommand{\MT}{\ensuremath{{\sf Tree}^{{\diamond}}}\xspace}
\newcommand{\md}{\ensuremath{{\sf md}}\xspace}
\newcommand{\J}{{\ensuremath{\mathsf{J}}}\xspace}
\newcommand{\la}{\langle}
\newcommand{\ra}{\rangle}
\newcommand{\h}{\ensuremath{{\sf h}}\xspace}
\newcommand{\T}{{\mathscr{T}}\xspace}
\newcommand{\F}{{\mathscr{F}}\xspace}

\newtheorem{theorem}{Theorem}[section]

\newtheorem{definition}[theorem]{Definition}
\newtheorem{lemma}[theorem]{Lemma}

\newtheorem{corollary}[theorem]{Corollary}
\newtheorem{proposition}[theorem]{Proposition}

\newtheorem{remark}[theorem]{Remark}

\newcounter{nc}

\title{A tree rewriting system for the Reflection Calculus}
\author{Sofía Santiago-Fernández\thanks{\url{sofia.santiago@ub.edu}} \and Joost J. Joosten\thanks{\url{jjoosten@ub.edu}} \and David Fernández-Duque\thanks{\url{fernandez-duque@ub.edu}}}
\date{}

\begin{document}

\maketitle

  \begin{abstract}
    The \textit{Reflection Calculus} ($\RC$, c.f. \cite{Beklemishev2012}, \cite{dashkov2012positive}) is the fragment of the polymodal logic \GLP in the language $\mathcal{L}^+$ whose formulas are built up from $\top$ and propositional variables using conjunction and diamond modalities. \RC is complete with respect to the arithmetical interpretation that associates modalities with reflection principles and has various applications in proof theory, specifically ordinal analysis.

    We present $\TRC$, a tree rewriting system (c.f. \cite{baader1998term}) that is adequate and complete with respect to $\RC$, designed to simulate $\RC$ derivations. $\TRC$ is based on a given correspondence between formulas of $\mathcal{L}^{+}$ and modal trees \MT. Modal trees are presented as an inductive type (c.f. \cite{fitting2002types}, \cite{pierce2002types}) allowing precise positioning and transformations which give rise to the formal definition of rewriting rules and facilitates formalization in proof assistants. Furthermore, we provide a rewrite normalization theorem for systematic rule application. The normalization of the rewriting process enhances proof search efficiency and facilitates implementation (c.f. \cite{wos1992automated}, \cite{goubault2001proof}, \cite{newborn2000automated}). 
    
    By providing $\TRC$ as an efficient provability tool for \RC, we aim to help on the study of various aspects of the logic such as the subformula property and rule admissibility.
  \end{abstract}

\section{Introduction}

Modal logics provide an attractive alternative to first or higher order logic for computational applications, largely due to the fact that they often enjoy a decidable consequence relation while remaining expressive enough to describe intricate processes.
However, decidability alone does not suffice for practical implementation when complexity comes at a hefty price tag; even propositional logic is {\sc np}-complete, which quickly becomes intractable as formula size and especially the number of variables is large.

This is no longer an issue when working in {\em strictly positive} fragments (see e.g.~\cite{kikot2019kripke}), which in contrast enjoy a polynomially decidable consequence relation.
Strictly positive formulas do not contain negation and instead are built from atoms and $\top$ using conjunction and $\Diamond$ (or, more generally, a family of modalities $\langle i\rangle$ indexed by $i$ in some set $I$).
Strictly positive formulas tend to be contingent, so validity and satisfiability are no longer the most central problems, but the consequence relation is indeed useful for example for reasoning about ontologies and is the basis for some description logics~\cite{BaaderCalvaneseEtAl2007}.

One remarkable success story for strictly positive logics comes from the {\em reflection calculus} ($\bf RC$)~\cite{dashkov2012positive,Beklemishev2014Positive}.
Beklemishev has shown how Japaridze's polymodal provability logic $\bf GLP$~\cite{Japaridze:1988} can be used to perform a proof-theoretic analysis of Peano aritmetic and its fragments~\cite{Beklemishev:2004}; however, the logic $\bf GLP$ is notoriously difficult to work with, especially as it is not Kripke-complete.
In contrast, its strictly positive fragment is rather tame from both a theoretical and computational point of view, yet suffices for the intended proof-theoretic applications.

The current work is inspired by two distinct ideas that have arisen in the study of strictly positive logics.
The first is the tree representation of formulas, which yield a way to decide strictly positive implications.
This was developed by Kikot et al.~\cite{kikot2019kripke} in a general setting and by Beklemishev~\cite{Beklemishev2014Positive}
in the context of $\bf RC$.
In both cases, one assigns to each strictly positive formula $\varphi$ a finite, tree-like Kripke model $T(\varphi)$ with the crucial property that $\varphi\to\psi$ is valid if and only if $T(\varphi) \models \psi$.
Thus the study of strictly positive fragments can be reduced to the study of their tree-like Kripke models.

The second is the connection of strictly positive calculi to term rewrite systems.
Strictly positive formulas and, particularly, those built exclusively from $\top$ and the modalities $\langle i\rangle$, may be regarded as {\em words} (or `worms'). This has prompted Beklemishev~\cite{Beklemishev2018-BEKANO} to view strictly positive fragments as term-rewriting systems~\cite{baader1998term}, but connections between such systems and modal logic are not new and can be traced back to Foret~\cite{Foret1992}.

Term rewriting is a discipline that integrates elements of logic, universal algebra, automated theorem proving, and functional programming.
It has applications in algebra (e.g. Boolean algebra), recursion
theory (computability of rewriting rules), software engineering and programming languages
(especially functional and logic programming \cite{Terese03}), 
with the $\lambda$-calculus perhaps being the most familiar example~\cite{LambdaBook}. 
Of particular interest to us, tree rewriting systems~\cite{TreePaper} are term rewriting systems such that terms are trees.  

When terms represent formulas, rewrite rules are similar to deep inference rules, i.e.~rules which may be applied to strict subformulas of the displayed formulas.
This is the approach taken by Shamkanov~\cite{shamkanov2015nested} for developing a cut-free calculus for $\bf GLP$.
As is the case for other technical differences between $\bf GLP$ and the reflection calculus, our tree rewrite system makes up for the loss in expressive power with increased simplicity and transparent Kripke semantics.

Our approach is to recast $\bf RC$ as an abstract rewriting system in which terms are trees.
In the parlance of rewrite systems, cut-elimination can be viewed as a normalization procedure for derivations.
In our setting we do not have an analogue of the cut rule, but we do obtain a rewriting normalization theorem which states that the rewriting process can be consistently and efficiently executed by grouping rewriting rules by their kinds and applying them in a designated sequence. By doing so, it enhances our comprehension of the dynamics of the tree rewriting system, offering valuable insights into the nature of the rewriting process and the interplay among rules. Moreover, it furnishes an efficient framework for proof search methodologies. Thanks to the normalization theorem, the need for exhaustive exploration is minimized by focusing on normalized rewriting sequences, which mitigates the risk of redundancy in rewriting. Consequently, when searching for a proof of a certain result, we only need to consider the normalized derivations, thereby reducing the proof search space and improving computational efficiency~\cite{goubault2001proof}. Furthermore, it serves as a practical guide for implementing the rewriting process in proof assistants~\cite{newborn2000automated}.

In our presentation, we make use of the inductive type of lists within the framework of type theory (cf. \cite{pierce2002types}, \cite{fitting2002types}) to define the trees in our tree rewriting system. The use of lists allows to define inductive structures with an order, facilitating the specification of internal positions and transformations for rewriting systems, and its formalization in proof assistants.

Since lists play such a central role in our work, we conclude this introduction by establishing some notation. A list of elements of type $\mathcal{A}$ is either the empty list $\varnothing$ or $[x] \smallfrown L$ for $x$ an element of type $\mathcal{A}$, a list $L$ of elements of type $\mathcal{A}$ and $\smallfrown$ the operator of concatenation of lists. We write $x \smallfrown L$ and $L \smallfrown x$ to denote $[x] \smallfrown L$ and $L \smallfrown [x]$, respectively. The length of a list $L$ of elements of type $\mathcal{A}$ is denoted by $|L|$.

\section{From $\K^+ $ to $\RC$}

In this section we present the basic sequent-style system $\K^+$ for the language of strictly positive formulae, concluding by an introduction to the Reflection Calculus ($\RC$) as an extension of $\K^+$. 

We consider the language of strictly positive formulae $\mathcal{L}^+$ composed from propositional variables $p$,$q$,..., in ${\sf Prop}$, the constant $\top$, and connectives $\wedge$ for conjunction and $\la \alpha \ra$ for diamond modalities for each ordinal $\alpha \in \omega$. Formally, the strictly positive formulae $\varphi$ of $\mathcal{L}^+$ are generated by the following grammar: 
\begin{equation*}
\varphi ::= \top \hspace{0.1cm}|\hspace{0.1cm} p \hspace{0.1cm}|\hspace{0.1cm} \la \alpha \ra \varphi \hspace{0.1cm}|\hspace{0.1cm} (\varphi \wedge \varphi), \hspace{0.2cm} \alpha \in \omega \text{ and } p \in {\sf Prop}.
\end{equation*}

The \textit{modal depth} of $\varphi$, denoted by $\md(\varphi)$, is recursively defined as $\md(\top) := 0$, $\md(p) := 0$ for $p \in {\sf Prop}$, $\md(\la \alpha \ra \varphi) := \md(\varphi) + 1$ and $\md(\varphi \wedge \psi) := {\sf max}\{\md(\varphi),\md(\psi)\}$. 

\textit{Sequents} are expressions of the form $\varphi \vdash \psi$ for $\varphi, \psi \in \mathcal{L}^{+}$. If ${\sf L}$ is a logic, we write $\varphi \vdash_{\sf L} \psi$ for the statement that $\varphi \vdash \psi$ is provable in ${\sf L}$. We write $\varphi \equiv_{\sf L} \psi$ to denote $\varphi \vdash_{\sf L} \psi$ and $\psi \vdash_{\sf L} \varphi$.

Polymodal $\K$ can be readily adapted to its strictly positive variant, where most notably the necessitation rule is replaced by distribution for each diamond modality.

\begin{definition}($\K^+$, \cite{Beklemishev2018-BEKANO})
The basic sequent-style system $\K^+$ is given by the following axioms and rules:
\begin{description}
\item $\varphi \vdash_{\K^+} \varphi$; $\varphi \vdash_{\K^+} \top$;
\item if $\varphi \vdash_{\K^+} \psi$ and $\psi \vdash_{\K^+} \phi$ then $\varphi \vdash_{\K^+} \phi$ (cut);
\item $\varphi \wedge \psi \vdash_{\K^+} \varphi$ and $\varphi \wedge \psi \vdash_{\K^+} \psi$ (elimination of conjunction);
\item if $\varphi \vdash_{\K^+} \psi$ and $\varphi \vdash_{\K^+} \phi$ then $\varphi \vdash_{\K^+} \psi \wedge \phi$ (introduction to conjunction);
\item if $\varphi \vdash_{\K^+} \psi$ then $\langle \alpha \rangle \varphi \vdash_{\K^+} \langle \alpha \rangle \psi$ (distribution).
\end{description}
\end{definition}

For $\Pi$ a finite list of strictly positive formulae, $\bigwedge \Pi$ is defined by $\top$ for $\Pi = \varnothing$ and $\varphi \wedge \bigwedge \hat{\Pi}$ for $\Pi = \varphi \smallfrown \hat{\Pi}$. Note that $\bigwedge (\Pi_1 \smallfrown \Pi_2) \equiv_{\K^+} \bigwedge \Pi_1 \wedge \bigwedge \Pi_2$ for $\Pi_1$ and $\Pi_2$ finite lists of strictly positive formulae.

A diamond modality can be distributed over a conjunction of formulas for $\K^+$ as follows. 

\begin{lemma}
\label{RCconjdiamondn}
$\langle \alpha \rangle (\varphi_1 \land ... \land \varphi_{n}) \vdash_{\K^+} \langle \alpha \rangle \varphi_1 \land ... \land \langle \alpha \rangle \varphi_{n}$.
\end{lemma}

\begin{proof}
By an easy induction on $n$.
\end{proof}

We aim to define a tree rewriting system adequate and complete w.r.t. the Reflection Calculus, the strictly positive fragment of Japaridze's polymodal logic formulated as an extension of $\K^+$. 

\begin{definition}(\RC, \cite{Beklemishev2012}, \cite{dashkov2012positive})
The \textit{Reflection Calculus} (\RC) is the strictly positive modal logic extending $\K^+$ by the following axioms:
\begin{description}
\item $\langle \alpha \rangle \langle \alpha \rangle \varphi \vdash_{\RC} \langle \alpha \rangle \varphi$ (transitivity);
\item $\langle \alpha \rangle \varphi \vdash_{\RC} \langle \beta \rangle \varphi$, $\alpha > \beta$ (monotonicity);
\item $\la \alpha \ra \varphi \wedge \la \beta \ra \psi \vdash_{\RC} \la \alpha \ra (\varphi \wedge \la \beta \ra \psi)$, $\alpha > \beta$ (J).
\end{description}
\end{definition}

\section{Modal trees}

In this section we present modal trees, a concrete set of inductively defined trees on which our rewriting system is based, and the corresponding framework for their manipulation. Modal trees are finite labeled trees with nodes labeled with lists of propositional variables and edges labeled with ordinals less than $\omega$. Specifically, modal trees can be regarded as tree-like Kripke models of the form $(\mathcal{W}, \{R_\alpha\}_{\alpha \in \omega}, v)$ such that an ordinal $\alpha$ labels an edge if $R_\alpha$ relates the corresponding nodes, and a list of propositional variables labels a node if its elements are the only propositional variables being true under the valuation $v$ in that node. However, for technical convenience, both in presenting the rewrite system and in formalizing our results in a proof assistant, it will be convenient to present modal trees as inductively defined structures.
In particular, the children of a node of a modal tree are given by lists instead of sets, providing a default ordering on its children useful for unambiguously determining positions in the tree.

\begin{definition}(\MT) The set of modal trees \MT is defined recursively to be the set of pairs $\la \Delta; \Gamma \ra$ where $\Delta$ is a finite list of propositional variables and $\Gamma$ is a finite list of pairs $(\alpha, {\tt T})$, with $\alpha < \omega$ and ${\tt T} \in \MT$.
\end{definition} 

Elements of \MT will be denoted by ${\tt T} $ and ${\tt S}$. Note that we employ distinct notations to enhance clarity on wether a pair is a modal tree: $\langle \cdot ; \cdot \rangle$ is used for a pair representing a modal tree, while $(\cdot , \cdot)$ denotes a pair comprising an ordinal and a modal tree. The \textit{root} of a modal tree $\la \Delta ; \Gamma \ra$ is $\Delta$ and its \textit{children} is the list $[{\tt S} | (\alpha,{\tt S}) \in \Gamma]$. Note that, in general we write $[f(\alpha,{\tt S}) | (\alpha,{\tt S}) \in \Gamma]$ to denote the list $[f(\alpha_1,{\tt S}_1),...,f(\alpha_n,{\tt S}_n)]$ for $\Gamma = [(\alpha_1,{\tt S}_1),...,(\alpha_n,{\tt S}_n)]$, $n \geq 0$ and $f$ a function of domain ${\sf Ord}^{< \omega} \times \MT$. For the sake of readability we write $\gamma \in \Gamma$ and ${\tt T} \in \Gamma$ to denote $\gamma \in [\alpha | (\alpha, {\tt S}) \in \Gamma]$ and ${\tt T} \in [{\tt S} | (\alpha,{\tt S}) \in \Gamma]$ respectively, since the context permits a clear distinction. A modal tree is called a \textit{leaf} if it has an empty list of children. The \textit{height} of a modal tree ${\tt T}$, denoted by $\h({\tt T})$, is inductively defined as $\h(\la \Delta ; \varnothing \ra) := 0$ and $\h(\la \Delta ; \Gamma \ra) := {\sf max}[\h({\tt S}) | {\tt S} \in \Gamma] + 1$. 

The sum of modal trees is the tree obtained by concatenating their roots and children.

\begin{definition}
The \textit{sum} of modal trees ${\tt T}_1 = \la \Delta_1; \Gamma_{1}\ra$ and ${\tt T}_2 = \la \Delta_2; \Gamma_2\ra$ is defined as ${\tt T}_1 + {\tt T}_2 := \la\Delta_1 \smallfrown \Delta_2; \Gamma_1  \smallfrown \Gamma_2\ra$.    
\end{definition}

More generally, for $\Lambda$ a finite list of modal trees, $\sum \Lambda$ is defined as $\la \varnothing ; \varnothing \ra$ if $\Lambda = \varnothing$ and ${\tt T} + \sum \hat{\Lambda}$ if $\Lambda = {\tt T} \smallfrown \hat{\Lambda}$. Note that $\h({\tt T}_1 + {\tt T}_2) = {\sf max} \{\h({\tt T}_1), \h({\tt T}_2)\}$ for ${\tt T}_1,{\tt T}_2 \in \MT$.

A standard numbering of the nodes of the tree by strings of positive integers allows us to refer to positions in a tree. Specifically, the set of positions of a tree includes the root position, defined as the empty string, and the positions from its children which are obtained by appending the order of each child to its positions. 

\begin{definition}(Set of positions)
The set of positions of a modal tree ${\tt T} = \la \Delta ; \Gamma \ra$, denoted by ${\sf Pos}({\tt T}) \in \mathcal{P}(\mathbb{N}^*)$, is inductively defined as 
\begin{itemize}
    \item ${\sf Pos}(\la \Delta ; \varnothing \ra) := \{\epsilon\}$ for $\epsilon \in \mathbb{N}^*$ the empty string,
    \item ${\sf Pos}(\la \Delta ; \Gamma \ra) := \{ \epsilon \} \cup \bigcup\limits_{i = 1}^{n} \{i\mathbf{k} | \mathbf{k} \in {\sf Pos}({\tt S}_i)\}$ for $\Gamma = [(\alpha_1,{\tt S}_1), ... , (\alpha_n,{\tt S}_n)]$. 
\end{itemize}
\end{definition}

Using the precise position apparatus we can define derived notions like, for example, that of subtree.

\begin{definition}(Subtree)
The subtree of ${\tt T} \in \MT$ at a position $\mathbf{k} \in {\sf Pos}({\tt T})$, denoted by ${\tt T}|_{\mathbf{k}}$, is inductively defined over the length of $\mathbf{k}$ as
\begin{itemize}
    \item ${\tt T}|_{\epsilon} := {\tt T}$,
    \item ${\tt T}|_{i\mathbf{r}} := {\tt S}_i|_\mathbf{r}$ for $1 \leq i \leq n$ such that ${\tt T} = \la \Delta ; [(\alpha_1,{\tt S}_1), ... , (\alpha_n,{\tt S}_n)] \ra$.
\end{itemize}
\end{definition}

We can now define subtree replacement based on the precise positioning provided.

\begin{definition}(Replacement)
Let ${\tt T},{\tt S} \in \MT$ and $\mathbf{k} \in {\sf Pos}({\tt T})$. The \textit{tree obtained from ${\tt T}$ by replacing the subtree at position $\mathbf{k}$ by ${\tt S}$}, denoted by ${\tt T}[{\tt S}]_\mathbf{k}$, is inductively defined over the length of $\mathbf{k}$ as
\begin{itemize}
    \item ${\tt T}[{\tt S}]_\epsilon := {\tt S}$,
    \item ${\tt T}[{\tt S}]_{i\mathbf{r}} := \la \Delta ; [(\alpha_1,{\tt S}_1),...,(\alpha_i,{\tt S}_i[{\tt S}]_\mathbf{r}),...,(\alpha_n,{\tt S}_n)]\ra$ for $1 \leq i \leq n$ and \\${\tt T} = \la \Delta ; [(\alpha_1,{\tt S}_1),...,(\alpha_n,{\tt S}_n)] \ra$.
\end{itemize}
\end{definition}

Here below we present useful results on positioning and replacement in a modal tree.

\begin{lemma}
\label{LemmaSubtreeReplacement}
Let ${\tt T},{\tt S}, \hat{\tt S} \in \MT$ be modal trees. Then, for $\mathbf{k}$ and $\mathbf{r}$ belonging to the adequate position sets we have
\begin{enumerate}
    \item $({\tt T}|_\mathbf{k})|_\mathbf{r} = {\tt T}|_\mathbf{kr}$;
    \item ${\tt T} [{\tt T}|_\mathbf{k}]_\mathbf{k} = {\tt T}$;
    \item $({\tt T} [{\tt S}]_\mathbf{k})|_\mathbf{k} = {\tt S}$;
    \item $({\tt T}[\hat{\tt S}]_\mathbf{k})[{\tt S}]_\mathbf{k} = {\tt T}[{\tt S}]_\mathbf{k}$ \hspace{0.2cm}\textit{(transitivity of replacement)};
    \item $({\tt T}[\hat{\tt S}]_\mathbf{k})[{\tt S}]_\mathbf{kr} = {\tt T}[\hat{\tt S}[{\tt S}]_\mathbf{r}]_\mathbf{k}$.
\end{enumerate}
\end{lemma}

\begin{proof}
We proceed by induction on the tree structure of ${\tt T}$ for each statement. If ${\tt T}$ is a leaf, the results follow easily since $\mathbf{k} = \epsilon$. Otherwise, we continue by cases on the length of $\mathbf{k}$. For $\epsilon$ the statements trivially hold. Finally consider $i\mathbf{\hat{k}}$ for $1 \leq i \leq n$ and $\mathbf{\hat{k}} \in {\sf Pos}({\tt S}_i)$ such that ${\tt T} = \la \Delta ; [(\alpha_1,{\tt S}_1),...,(\alpha_n,{\tt S}_n)] \ra$. Then, by definition and each statement's inductive hypothesis for ${\tt S}_i$, we conclude
\begin{description}
\item[$1.$] $({\tt T}|_{i\mathbf{\hat{k}}})|_\mathbf{r} = ({\tt S}_i|_\mathbf{\hat{k}})|_\mathbf{r} = {\tt S}_i|_\mathbf{\hat{k}r} = {\tt T}|_{i\mathbf{\hat{k}r}}$;
\item[$2.$] ${\tt T} [{\tt T}|_{i\mathbf{\hat{k}}}]_{i\mathbf{\hat{k}}} = {\tt T} [{\tt S}_i|_{\mathbf{\hat{k}}}]_{i\mathbf{\hat{k}}} = \la \Delta ; [(\alpha_1,{\tt S}_1),...,(\alpha_i,{\tt S}_i [{\tt S}_i|_{\mathbf{\hat{k}}}]_{\mathbf{\hat{k}}}),...,(\alpha_n,{\tt S}_n)] \ra = {\tt T}$;
\item[$3.$] $({\tt T} [{\tt S}]_{i\mathbf{\hat{k}}})|_{i\mathbf{\hat{k}}} = (\la \Delta ; [(\alpha_1,{\tt S}_1),...,(\alpha_i,{\tt S}_i[{\tt S}]_\mathbf{\hat{k}}),...,(\alpha_n,{\tt S}_n)]\ra)|_{i\mathbf{\hat{k}}} = ({\tt S}_i[{\tt S}]_\mathbf{\hat{k}})|_\mathbf{\hat{k}} = {\tt S}$;
\item[$4.$]
\begin{equation*}
    \begin{split}
    ({\tt T}[\hat{\tt S}]_{i\mathbf{\hat{k}}})[{\tt S}]_{i\mathbf{\hat{k}}} & = (\la \Delta ; [(\alpha_1,{\tt S}_1),...,(\alpha_i,{\tt S}_i[\hat{\tt S}]_\mathbf{\hat{k}}),...,(\alpha_n,{\tt S}_n)] \ra)[{\tt S}]_{i\mathbf{\hat{k}}} \\
    & = \la \Delta ; [(\alpha_1,{\tt S}_1),...,(\alpha_i,({\tt S}_i[\hat{\tt S}]_\mathbf{\hat{k}})[{\tt S}]_\mathbf{\hat{k}}),...,(\alpha_n,{\tt S}_n)] \ra \\
    &  = \la \Delta ; [(\alpha_1,{\tt S}_1),...,(\alpha_i,{\tt S}_i[{\tt S}]_\mathbf{\hat{k}}),...,(\alpha_n,{\tt S}_n)] \ra = {\tt T}[{\tt S}]_{i\mathbf{\hat{k}}};
    \end{split}  
    \end{equation*}
\item[$5.$] \begin{equation*}
    \begin{split}
    ({\tt T}[\hat{\tt S}]_{i\mathbf{\hat{k}}})[{\tt S}]_{i\mathbf{\hat{k}r}} & = (\la \Delta ; [(\alpha_1,{\tt S}_1),...,(\alpha_i,{\tt S}_i[\hat{\tt S}]_\mathbf{\hat{k}}),...,(\alpha_n,{\tt S}_n)] \ra)[{\tt S}]_{i\mathbf{\hat{k}r}} \\
    & = \la \Delta ; [(\alpha_1,{\tt S}_1),...,(\alpha_i,({\tt S}_i[\hat{\tt S}]_\mathbf{\hat{k}})[{\tt S}]_\mathbf{\hat{k}r}),...,(\alpha_n,{\tt S}_n)] \ra \\
    &  = \la \Delta ; [(\alpha_1,{\tt S}_1),...,(\alpha_i,{\tt S}_i[\hat{\tt S}[{\tt S}]_\mathbf{r}]_\mathbf{\hat{k}}),...,(\alpha_n,{\tt S}_n)] \ra = {\tt T}[\hat{\tt S}[{\tt S}]_\mathbf{r}]_{i\mathbf{\hat{k}}}. 
    \end{split}  
    \end{equation*}
\end{description}
\end{proof}

\section{Relating formulas and modal trees}
\label{sec:Embeddings} 

Our tree rewriting system is based on a correspondence between the language of $\mathcal{L}^+$ and $\MT$. Thereby we can ensure that transformations within the structure of modal trees accurately simulate derivations in a considered proof system. For this purpose, we introduce the tree embedding operator $\T$ inductively defined over the set of strictly positive formulas mapping them to modal trees. This definition is inspired by the canonical tree representation of strictly positive formulae presented by Beklemishev (see \cite{Beklemishev2014Positive}) as a combinatiorial tool for proving the polytime decidability of \RC. Additionally, we define $\F$ mapping modal trees to formulas. Ultimately, we prove that composition $\T \circ \F$ serves as the identity over \MT, while $\F \circ \T$ acts as the identity on $\mathcal{L}^+$ modulo equality for $\K +$.

\begin{definition}($\T$)
The modal tree embedding is the function $\T : \mathcal{L}^{+} \longrightarrow \MT$ inductively defined over the structure of strictly positive formulae as 
\begin{itemize} 
    \item $\T(\top) := \la \varnothing;\varnothing \ra$,
    \item $\T (p) := \la [p]; \varnothing\ra$ for $p \in {\sf Prop}$,
    \item $\T (\la \alpha \ra \varphi) := \la \varnothing; [( \alpha, \T (\varphi))] \ra$ for $\varphi \in \mathcal{L}^{+}$,
    \item $\T (\varphi \wedge \psi) := \T (\varphi) + \T (\psi)$ for $\varphi, \psi \in \mathcal{L}^{+}$.
\end{itemize}
\end{definition}

\begin{figure}[H]
\centering
    \begin{subfigure}{.1\textwidth}
    \centering
    \begin{tikzpicture}     
    \node[state][label=right:{\small ($\varnothing$)}] (a) at (0,0) {};     
    \end{tikzpicture} 
    \caption{\small $\T (\top)$}
    \end{subfigure}
\hfill
    \begin{subfigure}{.1\textwidth}
    \centering
    \begin{tikzpicture}     
    \node[state][label=right:{\small $[p]$}] (a) at (0,0) {};      
    \end{tikzpicture} 
    \caption{\small $\T (p)$}
    \end{subfigure}
\hfill
    \begin{subfigure}{.3\textwidth}
    \centering
    \begin{tikzpicture}       
    \node[state] (a) at (0,0) {};  
    \node[state][label=right:{\small $\Delta_\varphi$}] (b) [below =of a] {}; 
    \node[emptystate] (c) [below =of b] {\small $\T (\varphi)$};
    \draw[dashed,-] (b) edge[dashed] (c) {};
    \draw[-latex] (a) -- (b) node[fill=white,inner sep=2pt,midway] {{\small $\alpha$}};  
    \end{tikzpicture} 
    \caption{{\small $\T (\la \alpha \ra \varphi)$}}
    \end{subfigure}
\hfill
\vspace{0.5cm}
    \begin{subfigure}{.3\textwidth}
    \centering
    \begin{tikzpicture}        
    \node[state][label=right:{\small $\Delta_{\varphi} \smallfrown \Delta_{\psi}$}] (a) at (0,0) {} ;
    \node[emptystate] (b) [below left =of a] {\small $\T (\varphi)$};
    \node[emptystate] (c) [below right =of a] {\small $\T (\psi)$};
    \draw[dashed,-] (a) edge[dashed] (b) ; 
    \draw[dashed,-] (a) edge[dashed] (c) ;
    \end{tikzpicture} 
    \caption{\small $\T (\varphi \wedge \psi)$}
    \end{subfigure}
\caption{Modal tree embedding}
\end{figure}
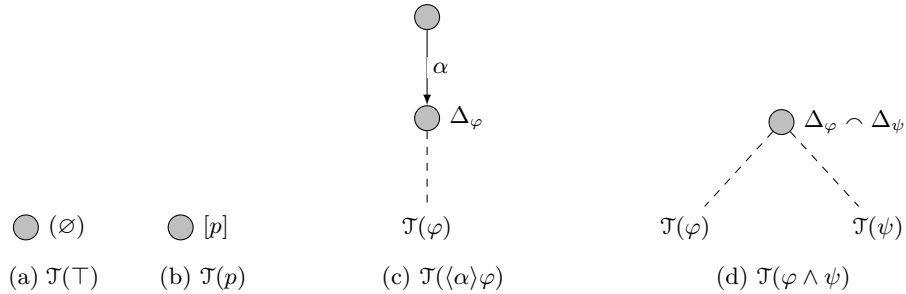

The modal depth of a formula coincides with the height of the modal tree which it is mapped to.

\begin{lemma}
\label{DepthMod}
    $\h(\T(\varphi)) = \md(\varphi)$ for $\varphi \in \mathcal{L}^{+}$.
\end{lemma}

\begin{proof}
    By an easy induction on the structure of $\varphi$.
\end{proof}

We also introduce a corresponding embedding in the opposite direction.

\begin{definition}($\F$)
The strictly positive formulae embedding is the function $\F : \MT \longrightarrow \mathcal{L}^{+}$ defined as 
\begin{equation*}
    \F(\la \Delta ; \Gamma \ra) := \bigwedge \Delta \wedge \bigwedge [\la \alpha \ra \F({\tt S}) | (\alpha,{\tt S}) \in \Gamma].
\end{equation*}
\end{definition}

For the sake of readability, we will simply write $\Diamond \Gamma$ to denote $[\la \alpha \ra \F({\tt S}) | (\alpha,{\tt S}) \in \Gamma]$. 

We conclude this section by providing a relation among strictly positive formulas and modal trees through the composition of the embeddings. 

\begin{proposition}[Embedding composition]
\label{PropInvEmbedding}
    $\T \circ \F = id_{\MT}$ and $\F \circ \T = id_{\mathcal{L}^{+} / \equiv_{\K^+}}$.
\end{proposition}

\begin{proof}
We firstly prove $\T \circ \F ({\tt T})= {\tt T}$ for ${\tt T} \in \MT$ by induction on the modal tree structure. The leaf case follows by definition since $\T (\F (\la\Delta; \varnothing \ra)) = \T(\bigwedge\Delta \wedge \top) = \la \Delta; \varnothing \ra$. Otherwise, assuming $\T(\F({\tt S})) = {\tt S}$ for every ${\tt S} \in \Gamma$, we conclude
\begin{equation*}
\begin{split}
\T(\F(\la \Delta; \Gamma\ra)) & = \T(\bigwedge\Delta \wedge \bigwedge \Diamond \Gamma) = \T(\bigwedge \Delta) + \T (\bigwedge \Diamond \Gamma) \\
 & = \la \Delta; \varnothing \ra + \sum [\T(\la \alpha \ra \F({\tt S})) | (\alpha,{\tt S}) \in \Gamma] \\
 &  = \la \Delta;\varnothing\ra + \sum [\la \varnothing; [( \alpha,\T(\F({\tt S})))] \ra | (\alpha,{\tt S}) \in \Gamma] \\
 & = \la \Delta ; \varnothing\ra + \sum [\la \varnothing; [(\alpha,{\tt S})] \ra | (\alpha,{\tt S}) \in \Gamma] = \la \Delta ; \Gamma\ra.
\end{split}
\end{equation*}

Finally, we prove $\F \circ \T (\varphi) \equiv_{\K^+} \varphi$ for $\varphi \in \mathcal{L}^{+}$ by induction on the structure of $\varphi$. 
\begin{itemize}
    \item $\F \circ \T (\top) = \top \wedge \top \equiv_{\K^+} \top$; $\F \circ \T (p) = (p \wedge \top) \wedge \top \equiv_{\K^+} p$.
    \item Let us assume $\F \circ \T (\psi)  \equiv_{\K^+} \psi$. Then,
    \begin{equation*}
        \F \circ \T (\la \alpha \ra \psi) = \F (\la \varnothing ; [(\alpha , \T (\psi))] \ra = \top \wedge \la \alpha \ra \F \circ \T (\psi) \wedge \top \equiv_{\K^+} \la \alpha \ra \psi.
    \end{equation*}
    \item Let $\T (\psi) = \la \Delta_\psi ; \Gamma_\psi \ra$ and $\T (\phi) = \la \Delta_\phi ; \Gamma_\phi \ra$. Assuming $\F \circ \T (\psi) \equiv_{\K^+} \psi$ and $\F \circ \T (\phi) \equiv_{\K^+} \phi$,
    \begin{equation*}
        \F \circ \T (\psi \wedge \phi) \equiv_{\K^+}
         \bigwedge \Delta_\psi \wedge \bigwedge \Diamond \Gamma_\psi \wedge \bigwedge \Delta_\phi \wedge \bigwedge \Diamond \Gamma_\phi \equiv_{\K^+} \psi \wedge \phi.
    \end{equation*}
\end{itemize}
\end{proof}

\section{The tree rewriting system for \RC}

We introduce the tree rewriting system for $\RC$, an abstract rewriting system for \MT which will be proven adequate and complete w.r.t. \RC in the next section. Additionally, we present useful results for the rewriting, along with the Inside Rewriting Property which involves transforming a subtree while preserving the remaining parts invariant.

An \textit{abstract rewriting system} is a pair $(A, \{\hookrightarrow^\mu\}_{\mu \in I})$ consisting of a set $A$ and a sequence of binary relations $\hookrightarrow^\mu$ on $A$, also called rewriting rules (r.r.). Instead of $(a,b) \in {\hookrightarrow^\mu}$ we write $a \hookrightarrow^\mu b$ and call \textit{$b$ is obtained by applying $\mu$ to $a$} or \textit{$b$ is obtained by performing one step $\mu$ to $a$}. The composition of rewriting rules $\mu_1$ and $\mu_2$ is written $a \hookrightarrow^{\mu_1} \circ \hookrightarrow^{\mu_2} b$ and denotes that there is $\hat{a} \in A$ such that $a \hookrightarrow^{\mu_1} \hat{a} \hookrightarrow^{\mu_2} b$.

In particular, the rewriting rules of the tree rewriting system for \RC transform a tree by replacing a subtree with a predetermined tree. The rules are classified into five kinds according to the performed transformation: atomic, structural, replicative, decreasing, and modal rewriting rules. Atomic rules duplicate or eliminate propositional variables in the lists labeling the nodes; the structural rule permutes the order of a node's children; the replicative rule duplicates a child of a node; decreasing rules either eliminate a child or remove a node and its children under certain conditions; and modal rules either decrease the label of an edge or apply a transformation simulating the $J$ axiom of $\RC$. 

To define the rewriting rules, we introduce the following notation. Let ${\tt T} = \la \Delta ; \Gamma \ra$ be a modal tree, $0< i,j \leq |\Gamma|$ and $n \leq |\Delta|$ such that $\Gamma = [(\alpha_1,{\tt S}_1),...,(\alpha_m,{\tt S}_m)]$. The $i$th element of $\Gamma$ is denoted by $\#_i \Gamma$. The list obtained by erasing the $i$th element of $\Gamma$, i.e. $[(\alpha_1,{\tt S}_1),...,(\alpha_{i-1},{\tt S}_{i-1}),(\alpha_{i+1},{\tt S}_{i+1})$\\$,...,(\alpha_m,{\tt S}_m)]$, is denoted by $\Gamma^{-i}$. The list obtained by duplicating the $i$th element of $\Gamma$ and placing it at the beginning, i.e. $(\alpha_i,{\tt S}_i) \smallfrown \Gamma$, is denoted by $\Gamma^{+i}$. Analogously, the list obtained by erasing the $n$th element of $\Delta$ and the list obtained by duplicating the $n$th element of $\Delta$ and placing it at the beginning are denoted by $\Delta^{-n}$ and $\Delta^{+n}$, respectively. The \textit{list obtained from $\Gamma$ by replacing its $i$th element} by $(\alpha,{\tt S})$ is defined by $\Gamma[(\alpha,{\tt S})]_i := [(\alpha_1,{\tt S}_1),...,(\alpha_{i-1},{\tt S}_{i-1}),(\alpha,{\tt S}),$\\$(\alpha_{i+1},{\tt S}_{i+1}), 
...,(\alpha_m,{\tt S}_m)]$. Note that we use the same notation for replacement in a list of pairs and replacement in a modal tree since the context allows for a clear distinction. Finally, the list obtained by interchanging the $i$th element with the $j$th element, i.e. $(\Gamma[\#_i \Gamma]_j)[\#_j \Gamma]_i$, is denoted by $\Gamma^{i \leftrightarrow j}$.

We can now present the tree rewriting system for Reflection Calculus.

\begin{definition}({\TRC})     
The \textit{tree rewriting system for \RC}, denoted by $\TRC$, is the abstract rewriting system $(\MT,\{\hookrightarrow^\mu\}_{\mu \in \mathscr{R}})$ for $\mathscr{R} = \{\rho^+,\rho^-,\sigma, \pi^+,\pi^-, \4, \lambda, \J\}$. The rewriting rules of $\mathscr{R}$ are defined in Table \ref{fig:rewrules}.
\end{definition}

\renewcommand{\arraystretch}{1.5}
\begin{table}
    \centering
\begin{tabularx}{1\textwidth}{| c | p{6.5cm} | c |}
\cline{1-3}
    Atomic r.r.  & \centering{${\tt T} \hookrightarrow^{\rho^+} {\tt T}[\la \Delta^{+i} ; \Gamma \ra]_\mathbf{k}$} & {\small $\rho^+$-rule} \\
    \cline{2-3}
    $0 < i \leq |\Delta|$ & \centering{${\tt T} \hookrightarrow^{\rho^-} {\tt T}[\la \Delta^{-i} ; \Gamma \ra]_\mathbf{k}$} & {\small $\rho^-$-rule} \\
\cline{1-3}
    Structural r.r. & \hfil\multirow{3}{*}{\centering ${\tt T} \hookrightarrow^{\sigma} {\tt T}[\la \Delta ; \Gamma^{i \leftrightarrow j} \ra]_\mathbf{k}$} & \multirow{3}{*}{{\small $\sigma$-rule}} \\
    $0 < i,j \leq |\Gamma|$, & & \\
    $i \neq j$ & & \\
\cline{1-3}
    Replicative r.r. & \hfil\multirow{2}{*}{\centering{${\tt T} \hookrightarrow^{\pi^+} {\tt T}[\la \Delta ; \Gamma^{+i} \ra]_\mathbf{k}$}} & \multirow{2}{*}{{\small $\pi^+$-rule}} \\
    $0 < i \leq |\Gamma|$ & & \\
\cline{1-3}
    \multirow{3}{*}{\shortstack[l]{Decreasing r.r. \\[5pt] \hfil $0 < i \leq |\Gamma|$ \\[5pt] \hfil $0 < j \leq |\tilde{\Gamma}|$}}  & \centering{${\tt T} \hookrightarrow^{\pi^-} {\tt T}[\la \Delta ; \Gamma^{-i} \ra]_\mathbf{k}$} & {\small $\pi^-$-rule} \\
    \cline{2-3}
    & \centering{${\tt T} \hookrightarrow^{\4} {\tt T}[\la \Delta ; \Gamma[(\beta, {\tt S})]_i \ra]_\mathbf{k}$} & \multirow{2}{*}{{\small $\4$-rule}} \\
    & \centering{for $\#_i \Gamma = (\beta, \la \tilde{\Delta} ; \tilde{\Gamma} \ra)$ and $\#_j \tilde{\Gamma}= (\beta , {\tt S})$} & \\
\cline{1-3}
    Modal r.r.  & \centering{${\tt T} \hookrightarrow^{\lambda} {\tt T}[\la\Delta; \Gamma[(\beta,{\tt S})]_i \ra]_\mathbf{k}$} & \multirow{2}{*}{{\small $\lambda$-rule}} \\
    $\alpha > \beta$ & \centering{for $\#_i \Gamma = (\alpha , {\tt S})$} & \\
    \cline{2-3}
    $0 < i,j \leq |\Gamma|$ & \centering{${\tt T} \hookrightarrow^{\J} {\tt T}[\la \Delta ; (\Gamma[(\alpha, \la \tilde{\Delta} ; \tilde{\Gamma} \smallfrown (\beta , {\tt S}) \ra)]_i)^{-j} \ra]_\mathbf{k}$} & \multirow{2}{*}{\small $\J$-rule} \\
    $i \neq j$ & \centering{for $\#_i \Gamma = (\alpha , \la \tilde{\Delta} ; \tilde{\Gamma} \ra)$ and $\#_j \Gamma = (\beta, {\tt S})$} & \\
\cline{1-3}
\end{tabularx}
    \caption{Rewriting rules of $\mathscr{R}$ for ${\tt T} \in \MT$, $\mathbf{k} \in {\sf Pos}({\tt T})$ and ${\tt T}|_\mathbf{k} = \la \Delta ; \Gamma \ra$.}
  \label{fig:rewrules}
\end{table}

Due to the transformations they induce, the rules are named as follows: the $\rho^+$-rule is called \textit{atom duplication}, the $\rho^-$-rule is called \textit{atom elimination}, the $\sigma$-rule is called \textit{child permutation}, the $\pi^+$-rule is called as \textit{child duplication}, the $\pi^-$-rule is called \textit{child elimination}, the $\4$-rule is called \textit{transitivity}, and the $\lambda$-rule is called \textit{monotonicity}. 

More generally, the tree rewriting relation is the union of the rewriting rules.

\begin{definition}($\hookrightarrow$)
The \textit{tree rewriting relation} $\hookrightarrow$ is defined as 
\begin{equation*}
\begin{split}
    \hookrightarrow & := \hspace{0.1cm}\hookrightarrow^{\rho^+} \cup \hookrightarrow^{\rho^-} \cup \hookrightarrow^{\sigma} \cup \hookrightarrow^{\pi^+} \cup \hookrightarrow^{\pi^-} \cup \hookrightarrow^{\4} \cup \hookrightarrow^{\lambda} \cup \hookrightarrow^{\J}.
\end{split}
\end{equation*} 
\end{definition}

We say that \textit{the step in} ${\tt T} \hookrightarrow {\tt T'}$ \textit{has been performed at position $\mathbf{k}$} if the applied rule replaces the subtree at position $\mathbf{k} \in {\sf Pos}({\tt T})$. We say that \textit{${\tt T}$ rewrites to ${\tt T'}$}, denoted by ${\tt T} \hookrightarrow^{*} {\tt T'}$, if ${\tt T'}$ is the result of applying zero, one or several rewriting rules of $\mathscr{R}$ to ${\tt T}$. In other words, $\hookrightarrow^{*}$ denotes the reflexive transitive closure of $\hookrightarrow$. The trees \textit{${\tt T}$ and ${\tt T'}$ are $\TRC$-equivalent}, denoted by ${\tt T} \overset{*}{\leftrightarrow} {\tt T'}$, if ${\tt T} \hookrightarrow^{*} {\tt T'}$ and ${\tt T'} \hookrightarrow^{*} {\tt T}$. If ${\tt T}$ rewrites to ${\tt T'}$ by applying the rewriting rule $\mu$ zero, one or several times, we write ${\tt T} \hookrightarrow^{\mu *} {\tt T'}$. For $\Omega$ a list of rewriting rules, we define ${\tt T} \hookrightarrow^{\Omega} {\tt S}$ inductively as ${\tt T} \hookrightarrow^{*} {\tt T}$ by applying no rewriting rules if $\Omega = \varnothing$, and ${\tt T} \hookrightarrow^{\mu} \circ \hookrightarrow^{\hat{\Omega}} {\tt S}$ for $\Omega = \mu \smallfrown \hat{\Omega}$. Likewise, for $\Omega_1$ and $\Omega_2$ lists of rewriting rules, ${\tt T} \hookrightarrow^{\Omega_1} \circ \hookrightarrow^{\Omega_2} {\tt S}$ denotes that there is $\hat{\tt S}$ such that ${\tt T} \hookrightarrow^{\Omega_1} \hat{\tt S} \hookrightarrow^{\Omega_2} {\tt S}$.

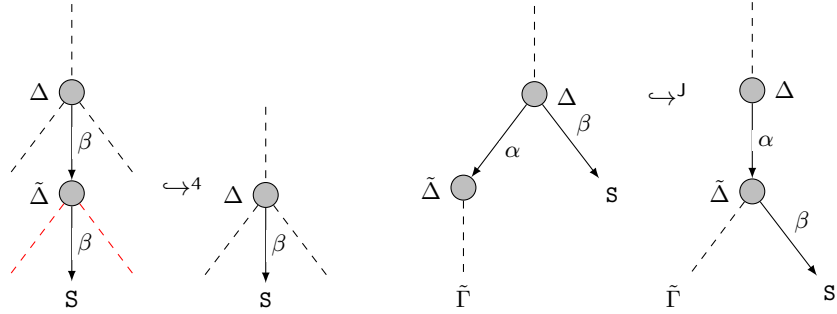
\begin{figure}[t]
\centering  
    \begin{subfigure}{0.4\textwidth} 
        \begin{subfigure}{0.4\textwidth}
            \begin{tikzpicture}   
            \node[emptystate] (a) at (0,0) {};
            \node[state][label=left: {\small $\Delta$}] (b) [below =of a] {};
            \node[emptystate][label=right:  $\hookrightarrow^{\4}$] (c) [below right =of b] {};
            \node[state][label=left: {\small $\tilde{\Delta}$}] (d) [below =of b] {};
            \node[emptystate] (e) [below right =of d] {};
            \node[emptystate] (f) [below =of d] {${\tt S}$};
            \node[emptystate] (g) [below left =of b] {};
            \node[emptystate] (h) [below left =of d] {};
            \draw[dashed,-] (a) edge[dashed] (b) {}; 
            \draw[dashed,-] (b) edge[dashed] (c) {};
            \draw[dashed,-] (b) edge[dashed] (g) {};
            \draw[red,dashed,-] (d) edge[dashed] (e) {}; 
            \draw[red,dashed,-] (d) edge[dashed] (h) {};
            \draw[-latex] (b) -- (d) node[fill=white,inner sep=2pt,midway] {\small $\beta$}; 
            \draw[-latex] (d) -- (f) node[fill=white,inner sep=2pt,midway] {\small $\beta$};
            \end{tikzpicture} 
        \end{subfigure}
        \hspace{0.4cm}
        \begin{subfigure}{0.4\textwidth}
            \begin{tikzpicture}   
            \node[emptystate] (a) at (0,0) {};
            \node[state][label=left: {\small $\Delta$}] (b) [below =of a] {};
            \node[emptystate] (c) [below right =of b] {};
            \node[emptystate] (d) [below =of b] {\small ${\tt S}$};
            \node[emptystate] (e) [below left =of b] {};
            \draw[dashed,-] (a) edge[dashed] (b) {}; 
            \draw[dashed,-] (b) edge[dashed] (c) {}; 
            \draw[dashed,-] (b) edge[dashed] (e) {};
            \draw[-latex] (b) -- (d) node[fill=white,inner sep=2pt,midway] {\small $\beta$}; 
            \end{tikzpicture}    
        \end{subfigure}
    \end{subfigure}
    \hspace{0.5cm}
    \begin{subfigure}{0.5\textwidth}
        \begin{subfigure}{0.4\textwidth}
            \begin{tikzpicture}  
            \node[emptystate] (a) at (0,0) {};
            \node[state][label=right: {\small $\Delta$} \hspace{0.7cm} $\hookrightarrow^{\J}$] (b) [below =of a] {};
            \node[state][label=left: {\small $\tilde{\Delta}$}] (c) [below left =of b] {};
            \node[emptystate] (d) [below right =of b] {\small ${\tt S}$};
            \node[emptystate] (e) [below =of c] {\small $\tilde{\Gamma}$};
            \draw[-latex] (b) -- (c) node[fill=white,inner sep=2pt,midway] {{\small $\alpha$}};
            \draw[-latex] (b) -- (d) node[fill=white,inner sep=2pt,midway] {{\small $\beta$}};
            \draw[dashed,-] (a) edge[dashed] (b);
            \draw[dashed,-] (c) edge[dashed] (e); 
            \end{tikzpicture}
        \end{subfigure}
        \hspace{0.6cm}
        \begin{subfigure}{0.3\textwidth}
            \centering
            \begin{tikzpicture} 
            \node[emptystate] (a) at (0,0) {};
            \node[state][label=right: {\small $\Delta$}] (b) [below =of a] {};
            \node[state][label=left: {\small $\tilde{\Delta}$}] (c) [below =of b] {};
            \node[emptystate](d) [below right =of c] {\small ${\tt S}$};
            \node[emptystate] (e) [below left =of c] {\small $\tilde{\Gamma}$};
            \draw[-latex] (b) -- (c) node[fill=white,inner sep=2pt,midway] {{\small $\alpha$}};
            \draw[-latex] (c) -- (d) node[fill=white,inner sep=2pt,midway] {{\small $\beta$}};
            \draw[dashed,-] (a) edge[dashed] (b);
            \draw[dashed,-] (c) edge[dashed] (e); 
            \end{tikzpicture}
        \end{subfigure} 
    \end{subfigure}
\caption{$\4$-rule and $\J$-rule.}
\label{fig:4Jrules}
\end{figure} 

Modal trees with permuted lists labeling the nodes are $\TRC$-equivalent. 

\begin{lemma}
\label{PermutationNode}
$\la \Delta_1 \smallfrown \Delta_2 ; \Gamma \ra \overset{*}{\leftrightarrow} \la \Delta_2 \smallfrown \Delta_1 ; \Gamma \ra$.
\end{lemma}

\begin{proof}
It suffices to show $\la \Delta_1 \smallfrown \Delta_2 ; \Gamma \ra \hookrightarrow^{*} \la \Delta_2 \smallfrown \Delta_1 ; \Gamma \ra$ by induction on the length of $\Delta_2$. If $\Delta_2$ is empty, the result trivially holds. Assuming $\la \Delta \smallfrown \Delta_2 ; \Gamma \ra \hookrightarrow^{*} \la \Delta_2 \smallfrown \Delta ; \Gamma \ra$ for any list of propositional variables $\Delta$, using atom duplication and atom elimination rewriting rules we conclude
\begin{equation*}
\begin{split}
\la \Delta_1 \smallfrown (p \smallfrown \Delta_2) ; \Gamma \ra & = \la (\Delta_1 \smallfrown p) \smallfrown \Delta_2 ; \Gamma \ra \hookrightarrow^{*} \la \Delta_2 \smallfrown (\Delta_1 \smallfrown p) ; \Gamma \ra \\
& \hookrightarrow^{\rho^{+}} \la p \smallfrown \Delta_2 \smallfrown \Delta_1 \smallfrown p ; \Gamma \ra \hookrightarrow^{\rho^-} \la (p \smallfrown \Delta_2) \smallfrown \Delta_1 ; \Gamma \ra.
\end{split}
\end{equation*}
\end{proof}

Here are some useful results on rewriting a sum of modal trees.

\begin{lemma}
\label{LemmaRewSum}
Let ${\tt T}_1,{\tt T}_2,{\tt T}_3,{\tt S}_1, {\tt S}_2 \in \MT$. Then,
\begin{enumerate}
    \item ${\tt T}_1 \overset{*}{\leftrightarrow} {\tt T}_1 + {\tt T}_1$;
    \item ${\tt T}_1 + {\tt T}_2 \overset{*}{\leftrightarrow} {\tt T}_2 + {\tt T}_1$;
    \item ${\tt T}_1 + {\tt T}_2 \hookrightarrow^{*} {\tt T}_1$ and ${\tt T}_1 + {\tt T}_2 \hookrightarrow^{*} {\tt T}_2$;
    \item If ${\tt T}_1 \hookrightarrow^{*} {\tt S}_1$, then ${\tt T}_1 + {\tt T}_2 \hookrightarrow^{*} {\tt S}_1 + {\tt T}_2$;
    \item If ${\tt T}_1 \hookrightarrow^{*} {\tt T}_2$ and ${\tt T}_1 \hookrightarrow^{*} {\tt T}_3$ then ${\tt T}_1 \hookrightarrow^{*} {\tt T}_2 + {\tt T}_3$;
    \item If ${\tt S}_1 \hookrightarrow^{*} {\tt T}_1$ and ${\tt S}_2 \hookrightarrow^{*} {\tt T}_2$, then ${\tt S}_1 + {\tt S}_2 \hookrightarrow^{*} {\tt T}_1 + {\tt T}_2$.
\end{enumerate}
\end{lemma}

\begin{proof}
The implication from left to right of the first statement holds by atom and child duplication, and the inverse implication by atom and child elimination.
The second result holds by Lemma \ref{PermutationNode} and child permutation. The third follows by atom and child elimination. The fourth result holds by induction on the number of rewriting steps performed in ${\tt T}_1 \hookrightarrow^{*} {\tt S}_1$ and by cases on the rewriting rules. The fifth statement holds by the fourth result using the statements one and two: 
\begin{equation*}
    {\tt T}_1 \hookrightarrow^{*} {\tt T}_1 + {\tt T}_1 \hookrightarrow^{*} {\tt T}_2 + {\tt T}_1 \hookrightarrow^{*} {\tt T}_1 + {\tt T}_2 \hookrightarrow^{*} {\tt T}_3 + {\tt T}_2 \hookrightarrow^{*} {\tt T}_2 + {\tt T}_3. 
\end{equation*}
Finally, by the fifth statement it suffices to show ${\tt S}_1 + {\tt S}_2 \hookrightarrow^{*} {\tt T}_1$ and ${\tt S}_1 + {\tt S}_2 \hookrightarrow^{*} {\tt T}_2$ to prove the sixth result, which follow by the third statement and the hypotheses.
\end{proof}

The following results state that transformations in subtrees can be extended to the entire tree, allowing for systematic and consistent modifications throughout the tree. In consequence, we can effectively manipulate and modify complex tree structures while leaving the other parts untouched.

\begin{proposition}[Inside Rewriting Property]
\label{InsideRewriting}   
If ${\tt S} \hookrightarrow^{*} {\tt S'}$, then ${\tt T} [{\tt S}]_\mathbf{k} \hookrightarrow^{*} {\tt T} [{\tt S'}]_\mathbf{k}$.
\end{proposition}

\begin{proof}
By induction on the number of rewriting steps that are performed in ${\tt S} \hookrightarrow^{*} {\tt S'}$. If no rewriting step is performed, the result is trivially satisfied. Now assume ${\tt T}[{\tt S}]_\mathbf{k} \hookrightarrow^{*} {\tt T}[\hat{\tt S}]_\mathbf{k}$ for ${\tt S} \hookrightarrow^{*} \hat{\tt S}$ by performing $n$ rewriting steps. Moreover, let ${\tt S} \hookrightarrow^{*} \hat{\tt S} \hookrightarrow^{\mu} \hat{\tt S}[\tilde{\tt S}]_\mathbf{r}$ for $\tilde{\tt S} \in \MT$ and $\mathbf{r} \in {\sf Pos}(\hat{\tt S})$ according to the rewriting rule $\mu$. Since ${\tt T}[{\tt S}]_\mathbf{k} \hookrightarrow^{*} {\tt T}[\hat{\tt S}]_\mathbf{k}$ by the inductive hypothesis, it suffices to show ${\tt T}[\hat{\tt S}]_\mathbf{k} \hookrightarrow^{\mu} {\tt T}[\hat{\tt S}[\tilde{\tt S}]_\mathbf{r}]_\mathbf{k} = ({\tt T}[\hat{\tt S}]_\mathbf{k})[\tilde{\tt S}]_\mathbf{kr}$ by Lemma \ref{LemmaSubtreeReplacement}. We proceed by induction on the tree structure of ${\tt T}$. For ${\tt T}$ a leaf, by the hypothesis,
\begin{equation*}
    {\tt T}[\hat{\tt S}]_\epsilon = \hat{\tt S} \hookrightarrow^{\mu} \hat{\tt S}[\tilde{\tt S}]_\mathbf{r} = ({\tt T}[\hat{\tt S}]_\epsilon)[\tilde{\tt S}]_{\epsilon \mathbf{r}}.
\end{equation*} 
Let ${\tt T} = \la \Delta ; [(\alpha_1,{\tt S}_1),...,(\alpha_m,{\tt S}_m)] \ra$ such that ${\tt S}_i[\hat{\tt S}]_\mathbf{l} \hookrightarrow^{\mu} ({\tt S}_i[\hat{\tt S}]_\mathbf{l})[\tilde{\tt S}]_{\mathbf{lr}}$ for $\mathbf{l} \in {\sf Pos}({\tt S}_i)$ and $1 \leq i \leq m$. We continue by cases on the length of $\mathbf{k}$. For $\mathbf{k} = \epsilon$, the result is trivially satisfied. For $\mathbf{k} = i\mathbf{\hat{k}}$ such that $\mathbf{\hat{k}} \in {\sf Pos}({\tt S}_i)$ and $1 \leq i \leq n$, we conclude by performing $\mu$ at position $\mathbf{kr}$ using the inductive hypothesis for ${\tt S}_i$ as follows,
\begin{equation*}
\begin{split}
{\tt T}[\hat{\tt S}]_{i\mathbf{\hat{k}}} & = \la \Delta ; [(\alpha_1,{\tt S}_1),..., (\alpha_i,{\tt S}_i[\hat{\tt S}]_{\mathbf{\hat{k}}}),...,(\alpha_m,{\tt S}_m)] \ra \\
 & \hookrightarrow^{\mu} \la \Delta ; [(\alpha_1,{\tt S}_1),..., (\alpha_i,({\tt S}_i[\hat{\tt S}]_{\mathbf{\hat{k}}})[\tilde{\tt S}]_{\mathbf{\hat{k}r}}),...,(\alpha_m,{\tt S}_m)] \ra = ({\tt T}[\hat{\tt S}]_{i\mathbf{\hat{k}}})[\tilde{\tt S}]_{i\mathbf{\hat{k}r}}.
\end{split}
\end{equation*}
\end{proof}

\begin{corollary}
\label{InsideRewritingRepl}
If ${\tt T}|_\mathbf{k} \hookrightarrow^{*} {\tt S}$, then ${\tt T} \hookrightarrow^{*} {\tt T}[{\tt S}]_\mathbf{k}$ for $\mathbf{k} \in {\sf Pos}({\tt T})$. 
\end{corollary}

\begin{proof}
By Proposition \ref{InsideRewriting} since ${\tt T} = {\tt T}[{\tt T}|_\mathbf{k}]_\mathbf{k}$ (Lemma \ref{LemmaSubtreeReplacement}).
\end{proof}

\section{Adequacy and completeness}

We aim to show how the tree rewriting system $\TRC$ faithfully simulates logical derivations in $\RC$ thanks to the embeddings defined in Section \ref{sec:Embeddings}. Thereby, adequacy and completeness theorems are presented as key results that underscore the efficacy of tree rewriting systems in relating logical inference and structural transformation.

Firstly, we show adequacy by proving that if a rewriting step is performed, the sequent of formulas which the trees are mapped to is provable in \RC.

\begin{proposition}
\label{SoundnessStep}
If ${\tt T} \hookrightarrow {\tt T'}$, then $\F({\tt T}) \vdash_{\RC} \F({\tt T'})$ for ${\tt T}, {\tt T'} \in \MT$.
\end{proposition}

\begin{proof}
By induction on the tree structure of ${\tt T}$. For ${\tt T}$ a leaf the result follows easily by cases on the performed rewriting rule. Now consider ${\tt T} = \la \Delta ; \Gamma \ra$ for $\Gamma = [(\alpha_1,{\tt S}_1),...,(\alpha_m,{\tt S}_m)]$ such that $\F({\tt S}_i) \vdash_\RC \F({\tt S'})$ if ${\tt S}_i \hookrightarrow {\tt S'}$ for $1 \leq i \leq m$. Assuming $\la \Delta ; \Gamma \ra \hookrightarrow^{\mu} {\tt T'}$ for $\mu$ a rewriting rule, we show $\F({\tt T}) \vdash_\RC \F({\tt T'})$ by cases on the length of the position at which $\mu$ is performed. 

First consider a rewriting at a position $i\mathbf{k} \in {\sf Pos}({\tt T})$. By definition, ${\tt T'}$ is of the form $\la \Delta ; [(\alpha_1,{\tt S}_1),...,(\alpha_i,{\tt S'}),...,(\alpha_m,{\tt S}_m)] \ra$ for ${\tt S}_i \hookrightarrow^{\mu} {\tt S'}$ by rewriting at position $\mathbf{k}$. Hence, by the inductive hypothesis, $\F({\tt S}_i) \vdash_\RC \F({\tt S'})$. Thus,
\begin{align*}
\F({\tt T}) & \equiv_{\K^+} \bigwedge \Delta \wedge \la \alpha_i \ra \F({\tt S}_i) \wedge \bigwedge \Diamond (\Gamma^{-i}) \\
& \vdash_\RC \bigwedge \Delta \wedge \la \alpha_i \ra \F({\tt S'}) \wedge \bigwedge \Diamond (\Gamma^{-i}) \equiv_{\K^+} \F({\tt T'}).
\end{align*}

For rewriting performed at position $\epsilon$, the proof concludes easily by cases on $\mu$. Let see in some detail the cases of transitivity and $\J$ rewriting rules.

\begin{description}
    \item[$\4$-rule:] Consider ${\tt T} \hookrightarrow^{\4} \la \Delta ; \Gamma[(\beta, {\tt S})]_i \ra$ such that $\#_i \Gamma = (\beta, \la \tilde{\Delta}; \tilde{\Gamma} \ra)$ for $0 < i \leq |\Gamma|$ and $\#_j \tilde{\Gamma} = (\beta , {\tt S})$ for $0 < j \leq |\tilde{\Gamma}|$. By Lemma \ref{RCconjdiamondn} and transitivity rule for $\RC$ we conclude,
    \begin{equation*}
    \begin{split}
    \F({\tt T}) & \equiv_{\K^+} \bigwedge \Delta \wedge \la \beta \ra \F(\la \tilde{\Delta} ; \tilde{\Gamma} \ra) \wedge \bigwedge \Diamond (\Gamma^{-i}) \\
    & \equiv_{\K^+} \bigwedge \Delta \wedge \la \beta \ra (\bigwedge \tilde{\Delta} \wedge \la \beta \ra \F({\tt S}) \wedge \bigwedge \Diamond (\tilde{\Gamma}^{-j})) \wedge \bigwedge \Diamond (\Gamma^{-i}) \\
    & \vdash_{\K^+} \bigwedge \Delta \wedge \la \beta \ra \la \beta \ra \F({\tt S}) \wedge \bigwedge \Diamond (\Gamma^{-i})\\
    & \vdash_\RC \bigwedge \Delta \wedge \la \beta \ra \F({\tt S}) \wedge \bigwedge \Diamond (\Gamma^{-i}) \\
    & \equiv_{\K^+} \F(\la \Delta ; \Gamma[(\beta, {\tt S})]_i \ra).
    \end{split}
    \end{equation*}
    \item[$\J$-rule:] Consider ${\tt T} \hookrightarrow^{\J} \la \Delta ; (\Gamma[(\alpha, \la \tilde{\Delta} ; \tilde{\Gamma} \smallfrown (\beta , {\tt S}) \ra)]_i)^{-j} \ra$ such that $\#_i \Gamma = (\alpha , \la \tilde{\Delta} ; \tilde{\Gamma} \ra)$ and $\#_j \Gamma = (\beta, {\tt S})$ for $0 < i,j, \leq |\Gamma|$ satisfying $i \neq j$. Let $i < j$ without loss of generality. By the $J$ rule for $\RC$ we conclude,
    \begin{equation*}
    \begin{split}
    \F({\tt T}) & \equiv_{\K^+} \bigwedge \Delta \wedge \la \alpha \ra \F(\la \tilde{\Delta} ; \tilde{\Gamma} \ra) \wedge \la \beta \ra \F({\tt S}) \wedge \bigwedge \Diamond ((\Gamma^{-j})^{-i}) \\
    & \vdash_\RC \bigwedge \Delta \wedge \la \alpha \ra (\F(\la \tilde{\Delta} ; \tilde{\Gamma} \ra) \wedge \la \beta \ra \F({\tt S})) \wedge \bigwedge \Diamond ((\Gamma^{-j})^{-i}) \\
    & \equiv_{\K^+} \bigwedge \Delta \wedge \la \alpha \ra (\F(\la \tilde{\Delta} ; \tilde{\Gamma} \smallfrown (\beta,{\tt S}) \ra)) \wedge \bigwedge \Diamond ((\Gamma^{-j})^{-i}) \\
    & \equiv_{\K^+} \F(\la \Delta ; (\Gamma[(\alpha, \la \tilde{\Delta} ; \tilde{\Gamma} \smallfrown (\beta , {\tt S})  \ra)]_i)^{-j} \ra).
    \end{split}
    \end{equation*}
\end{description}
\end{proof}

\begin{theorem}[Adequacy]
\label{Soundness}
If ${\tt T} \hookrightarrow^{*} {\tt T'}$, then $\F({\tt T}) \vdash_{\RC} \F({\tt T'})$ for ${\tt T}, {\tt T'} \in \MT$.
\end{theorem}

\begin{proof}
By an easy induction on the number of rewriting steps that are performed using Proposition \ref{SoundnessStep}.
\end{proof}

We conclude this section by showing that $\TRC$ is complete with respect to $\RC$.

\begin{theorem}[Completeness]
\label{Completeness}
If $\varphi \vdash_{\RC} \psi$, then $\T(\varphi) \hookrightarrow^{*} \T(\psi)$ for $\varphi, \psi \in \mathcal{L}^{+}$.
\end{theorem}

\begin{proof}
    By induction on the length of the \RC derivation. 
    
    \begin{description}
        \item If $\varphi \vdash_\RC \varphi$, then $\T(\varphi) \hookrightarrow^{*} \T(\varphi)$ by applying no rewriting rule; if $\varphi \vdash_\RC \top$, then $\T(\varphi) \hookrightarrow^{*} \T(\top)$ by atom and child elimination.
        \item[\textnormal{\textit{(Cut):}}] If $\varphi \vdash_\RC \psi$ and $\psi \vdash_\RC \phi$ such that $\T(\varphi) \hookrightarrow^{*} \T(\psi)$ and $\T(\psi) \hookrightarrow^{*} \T(\phi)$, then $\T(\varphi) \hookrightarrow^{*} \T(\phi) \hookrightarrow^{*} \T(\phi)$.
        \item[\textnormal{\textit{(Elimination of conjunction):}}] If $\varphi \wedge \psi \vdash_\RC \varphi$, by Lemma \ref{LemmaRewSum} follows $\T(\varphi \wedge \psi) = \T (\varphi) + \T(\psi) \hookrightarrow^{*} \T(\varphi)$. Analogously, if $\varphi \wedge \psi \vdash_\RC \psi$ then $\T(\varphi \wedge \psi) \hookrightarrow^{*} \T(\psi)$ by Lemma \ref{LemmaRewSum}.
        \item[\textnormal{\textit{(Introduction to conjunction):}}] If $\varphi \vdash_\RC \psi$ and $\varphi \vdash_\RC \phi$ such that $\T(\varphi) \hookrightarrow^{*} \T(\psi)$ and $\T(\varphi) \hookrightarrow^{*} \T(\phi)$, by Lemma \ref{LemmaRewSum} we conclude $\T(\varphi) \hookrightarrow^{*} \T(\psi) + \T(\phi) = \T(\psi \wedge \phi)$.
        \item[\textnormal{\textit{(Distribution):}}] If $\varphi \vdash_\RC \psi$ such that $\T(\varphi) \hookrightarrow^{*} \T(\psi)$, then $\T(\la \alpha \ra \varphi) \hookrightarrow^{*} \T(\la \alpha \ra \psi)$ by the inside rewriting property (Corollary \ref{InsideRewritingRepl}).
        \item[\textnormal{\textit{(Transitivity):}}] If $\la \alpha \ra \la \alpha \ra \varphi \vdash_\RC \la \alpha \ra \varphi$, by the $\4$-rule, 
        \begin{equation*}
            \T(\la \alpha \ra \la \alpha \ra \varphi) = \la \varnothing ; [(\alpha,\la \varnothing;[(\alpha,\T(\varphi))]\ra)]\ra \hookrightarrow^{\4} \la \varnothing;[(\alpha,\T(\varphi))] \ra = \T(\la \alpha \ra \varphi).
        \end{equation*}
        \item[\textnormal{\textit{(Monotonicity):}}] If $\la \alpha \ra \varphi \vdash_\RC \la \beta \ra \varphi$ for $\alpha > \beta$, by the $\lambda$-rule, 
        \begin{equation*}
            \T(\la \alpha \ra \varphi) = \la \varnothing ; [(\alpha, \T(\varphi))] \ra \hookrightarrow^{\lambda} \la \varnothing ; [(\beta, \T(\varphi))] \ra = \T(\la \beta \ra \varphi).
        \end{equation*}
        \item[\textnormal{\textit{(J):}}] Consider $\alpha > \beta$ and $\T(\varphi) = \la \Delta ; \Gamma \ra$. If $\la \alpha \ra \varphi \wedge \la \beta \ra \psi \vdash_\RC \la \alpha \ra (\varphi \wedge \la \beta \ra \psi)$, by the $\J$-rule,
        \begin{equation*}
        \begin{split}
        \T(\la \alpha \ra \varphi \wedge \la \beta \ra \psi) & = \la \varnothing;[(\alpha, \la \Delta ; \Gamma \ra),(\beta,\T(\psi))]\ra \hookrightarrow^{\J} \la \varnothing;[(\alpha,\la \Delta ; \Gamma \smallfrown (\beta,\T(\psi)) \ra)]\ra \\
        & = \la \varnothing;[(\alpha,\T(\varphi) + \T(\la \beta \ra \psi))] \ra = \T(\la \alpha \ra (\varphi \wedge \la \beta \ra \psi)).
        \end{split}    
        \end{equation*}
    \end{description}
\end{proof}

This concluding corollary states that a sequence can be shown in $\RC$ by identifying transformations within the trees in which the involved formulas are embedded into. Likewise, a rewriting can be proven by showing the corresponding $\RC$ sequence for the formulas embedding the corresponding trees.

\begin{corollary}
\label{TRStoRC}
$\T(\varphi) \hookrightarrow^{*} \T(\psi) \Longrightarrow \varphi \vdash_\RC \psi$ and $\F({\tt T}) \vdash_\RC \F({\tt T'}) \Longrightarrow {\tt T} \hookrightarrow^{*} {\tt T'}$.
\end{corollary}

\begin{proof}
By adequacy and completeness (Theorems \ref{Soundness} and \ref{Completeness}, respectively) using the embedding composition property (Proposition \ref{PropInvEmbedding}). 
\end{proof}

\section{The Rewrite Normalization Theorem}
\label{sec: Normalization}

In this section we present our rewriting normalization theorem, which allows us to perform rewriting in a designated sequence of rewriting rules according to their kinds. Specifically, we can always perform the rewrite process using normal rewriting sequences.

\begin{definition}(Normal rewrite)
A list of rewriting rules $\Omega$ is a {\em normal rewriting sequence} if it is of the form 
\begin{equation*}
    \Omega_{\pi^{+}} \smallfrown \Omega_{\diamond} \smallfrown \Omega_\delta \smallfrown \Omega_\rho \smallfrown \Omega_\sigma
\end{equation*} 
for $\Omega_{\pi^+}, \Omega_\diamond, \Omega_\delta, \Omega_\rho$ and $\Omega_\sigma$ lists of replicative, modal, decreasing, atomic and structural rewriting rules, respectively.
\end{definition}

The order of the presented normal rewriting sequence adheres to the following principles. Firstly, performing any kind of rewriting rule before a replicative one cannot be equivalently reversed. Similarly, the order of modal and decreasing rewriting rules cannot be interchanged (for example, it may be necessary to decrease the label of an edge in order to apply $\4$-rule). Finally, atomic and structural rewriting rules, which pertain to node labels and child permutation, are placed last because nodes and children may be removed during the rewriting process.

In the following, we disclose the results needed to establish the normal rewriting theorem. Their proofs delve into technical intricacies, requiring meticulous attention to the rewritten positions and the arguments determining the application of the concerned rewriting rules. If two rewriting rules are applied in different positions of the tree such that their outcome transformation is disjoint, their permutation is proven straightforward. However, if their application intersects, additional transformations may be required. To illustrate these scenarios, we provide simplified examples alongside the proofs.

\begin{lemma}
\label{PermutationWithRest}
If ${\tt T} \hookrightarrow^{\sigma} \circ \hookrightarrow^{\mu} {\tt S}$, then ${\tt T} \hookrightarrow^{\mu} \circ \hookrightarrow^{\sigma *} {\tt S}$ for $\mu \in \mathscr{R} \setminus \{\sigma\}$.
\end{lemma}

\begin{proof}
By cases on $\mu$. Each case is proven by induction on the length of the position of $\sigma$ and by cases on the position of $\mu$.
\end{proof}

\begin{corollary}
\label{CorollarySigma}
If ${\tt T} \hookrightarrow^{\sigma *} \circ \hookrightarrow^{\Omega} {\tt S}$, then ${\tt T} \hookrightarrow^{\Omega} \circ \hookrightarrow^{\sigma *} {\tt S}$ for $\Omega$ a list of atomic, replicative, decreasing and modal rewriting rules.
\end{corollary}

\begin{proof}
By induction on the length of $\Omega$ and by induction on the number of $\sigma$-steps performed using Lemma \ref{PermutationWithRest}.
\end{proof}

\begin{lemma}
\label{AtomicWithRest}
Let $\rho$ be an atomic rule. If ${\tt T} \hookrightarrow^{\rho} \circ \hookrightarrow^{\mu} {\tt S}$, then ${\tt T} \hookrightarrow^{\mu} \circ \hookrightarrow^{\rho *} {\tt S}$ for $\mu \in \{\pi^{+}, \pi^{-}, \4, \lambda, \J\}$.
\end{lemma}

\begin{proof}
By cases on $\mu$. Each case is shown by induction on the length of the position of $\rho$ and by cases on the position of $\mu$. If $\mu$ is $\pi^+$, we prove that either ${\tt T} \hookrightarrow^{\pi^+} \circ \hookrightarrow^{\rho} {\tt S}$ or ${\tt T} \hookrightarrow^{\pi^+} \circ \hookrightarrow^{\rho} \circ \hookrightarrow^{\rho}{\tt S}$ holds (see Figure \ref{fig:Pi+Normalization}). If $\mu$ is a decreasing rewriting rule, then either ${\tt T} \hookrightarrow^{\mu} {\tt S}$ or ${\tt T} \hookrightarrow^{\mu} \circ \hookrightarrow^{\rho} {\tt S}$. Otherwise we show ${\tt T} \hookrightarrow^{\mu} \circ \hookrightarrow^{\rho} {\tt S}$.
\end{proof}

\begin{corollary}
\label{CorollaryAtomic}
Let $\Omega_\rho$ be a list of atomic rewriting rules. If ${\tt T} \hookrightarrow^{\Omega_\rho} \circ \hookrightarrow^{\mu} {\tt S}$, then ${\tt T} \hookrightarrow^{\mu} \circ \hookrightarrow^{\Omega'_\rho} {\tt S}$ for $\mu \in \{\pi^+,\pi^-,\4,\lambda, \J\}$ and $\Omega'_\rho$ a list of atomic rewriting rules.
\end{corollary}

\begin{proof}
By an easy induction on the length of $\Omega_\rho$ using Lemma \ref{AtomicWithRest}.
\end{proof}

\begin{figure}[H]
    \centering
\begin{subfigure}{.7\textwidth}
\begin{subfigure}{.3\textwidth}
    \centering
    \begin{tikzpicture}   
    \node[state][label=right: \hspace{0.8cm}$\hookrightarrow^{\mu}$] (a) at (0,0) {} ;
    \node[emptystate] (b) [below left =of a] {};
    \node[emptystate] (c) [below right=of a] {${\tt S}$};
    \draw[dashed,-] (a) edge[dashed] (b); 
    \draw[-latex] (a) -- (c) node[fill=white,inner sep=2pt,midway] {{\small $\alpha$}};
    \end{tikzpicture} 
\end{subfigure}
\begin{subfigure}{.3\textwidth}
    \centering
    \begin{tikzpicture}
    \node[state][label=right: \hspace{1cm}$\hookrightarrow^{\pi^+}$] (a) at (0,0) {} ;
    \node[emptystate] (b) [below left =of a] {};
    \node[emptystate] (c) [below right=of a] {$\hat{\tt S}$};
    \draw[dashed,-] (a) edge[dashed] (b); 
    \draw[-latex] (a) -- (c) node[fill=white,inner sep=2pt,midway] {{\small $\alpha$}};
    \end{tikzpicture} 
\end{subfigure}
\begin{subfigure}{.3\textwidth}
    \centering
    \begin{tikzpicture}        
    \node[state] (a) at (0,0) {} ;
    \node[emptystate] (b) [below left =of a] {$\hat{\tt S}$};
    \node[emptystate] (c) [below right=of a] {$\hat{\tt S}$};
    \node[emptystate] (d) [below =of a] {};
    \draw[dashed,-] (a) edge[dashed] (d); 
    \draw[-latex] (a) -- (b) node[fill=white,inner sep=2pt,midway] {{\small $\alpha$}};
    \draw[-latex] (a) -- (c) node[fill=white,inner sep=2pt,midway] {{\small $\alpha$}};
    \end{tikzpicture} 
\end{subfigure}
\caption{${\tt T} \hookrightarrow^{\mu} \circ \hookrightarrow^{\pi^+} {\tt S}$ (for ${\tt S} \hookrightarrow^{\mu} \hat{\tt S}$).}
\end{subfigure}
\begin{subfigure}{1.1\textwidth}
\begin{subfigure}{.2\textwidth}
    \centering
    \begin{tikzpicture}        
    \node[state][label=right: \hspace{1cm}$\hookrightarrow^{\pi^+}$] (a) at (0,0) {} ;
    \node[emptystate] (b) [below left =of a] {};
    \node[emptystate] (c) [below right=of a] {${\tt S}$};
    \draw[dashed,-] (a) edge[dashed] (b); 
    \draw[-latex] (a) -- (c) node[fill=white,inner sep=2pt,midway] {{\small $\alpha$}};
    \end{tikzpicture} 
\end{subfigure}
\begin{subfigure}{.2\textwidth}
    \centering
    \begin{tikzpicture}        
    \node[state][label=right: \hspace{1cm}$\hookrightarrow^{\mu}$] (a) at (0,0) {} ;
    \node[emptystate] (b) [below left =of a] {${\tt S}$};
    \node[emptystate] (c) [below right=of a] {${\tt S}$};
    \node[emptystate] (d) [below =of a] {};
    \draw[dashed,-] (a) edge[dashed] (d); 
    \draw[-latex] (a) -- (b) node[fill=white,inner sep=2pt,midway] {{\small $\alpha$}};
    \draw[-latex] (a) -- (c) node[fill=white,inner sep=2pt,midway] {{\small $\alpha$}};
    \end{tikzpicture} 
\end{subfigure}
\begin{subfigure}{.2\textwidth}
    \centering
    \begin{tikzpicture}        
    \node[state][label=right: \hspace{1cm}$\hookrightarrow^{\mu}$] (a) at (0,0) {} ;
    \node[emptystate] (b) [below left =of a] {$\hat{\tt S}$};
    \node[emptystate] (c) [below right=of a] {${\tt S}$};
    \node[emptystate] (d) [below =of a] {};
    \draw[dashed,-] (a) edge[dashed] (d); 
    \draw[-latex] (a) -- (b) node[fill=white,inner sep=2pt,midway] {{\small $\alpha$}};
    \draw[-latex] (a) -- (c) node[fill=white,inner sep=2pt,midway] {{\small $\alpha$}};
    \end{tikzpicture} 
\end{subfigure}
\begin{subfigure}{.2\textwidth}
    \centering
    \begin{tikzpicture}        
    \node[state] (a) at (0,0) {} ;
    \node[emptystate] (b) [below left =of a] {$\hat{\tt S}$};
    \node[emptystate] (c) [below right=of a] {$\hat{\tt S}$};
    \node[emptystate] (d) [below =of a] {};
    \draw[dashed,-] (a) edge[dashed] (d); 
    \draw[-latex] (a) -- (b) node[fill=white,inner sep=2pt,midway] {{\small $\alpha$}};
    \draw[-latex] (a) -- (c) node[fill=white,inner sep=2pt,midway] {{\small $\alpha$}};
    \end{tikzpicture} 
\end{subfigure}
\caption{${\tt T} \hookrightarrow^{\pi^+} \circ \hookrightarrow^{\mu} \circ \hookrightarrow^{\mu} {\tt S}$.}
\end{subfigure}
\caption{${\tt T} \hookrightarrow^{\mu} \circ \hookrightarrow^{\pi^+} {\tt S} \Longrightarrow {\tt T} \hookrightarrow^{\pi^+} \circ \hookrightarrow^{\mu} \circ \hookrightarrow^{\mu} {\tt S}, \mu \in \mathscr{R} \setminus \{\pi^+\}$.}
\label{fig:Pi+Normalization}
\end{figure}

\begin{proposition}
\label{DecreasingWithModal}
Let $\delta$ be a decreasing rule. If ${\tt T} \hookrightarrow^{\delta} \circ \hookrightarrow^{\mu} {\tt S}$, then ${\tt T} \hookrightarrow^{\mu *} \circ \hookrightarrow^{\delta *} {\tt S}$ for $\mu \in \{\pi^{+}, \lambda, \J\}$. 
\end{proposition}

\begin{proof}
By induction on the length of the position of $\delta$ and by cases on the length of the position of $\mu$. 

Let $\delta$ be $\pi^-$. If $\mu$ is $\pi^+$, we show that either ${\tt T} \hookrightarrow^{\pi^+} \circ \hookrightarrow^{\pi^-} {\tt S}$ or ${\tt T} \hookrightarrow^{\pi^+} \circ \hookrightarrow^{\pi^-} \circ \hookrightarrow^{\pi^-}{\tt S}$ holds (see Figure \ref{fig:Pi+Normalization}). Otherwise we prove ${\tt T} \hookrightarrow^{\mu} \circ \hookrightarrow^{\pi^-} {\tt S}$.

Let $\delta$ be $\4$. If $\mu$ is $\pi^{+}$ we show that either ${\tt T} \hookrightarrow^{\pi^+} \circ \hookrightarrow^{\4} {\tt S}$ or ${\tt T} \hookrightarrow^{\pi^+} \circ \hookrightarrow^{\4} \circ \hookrightarrow^{\4}{\tt S}$ (see Figure \ref{fig:Pi+Normalization}). For $\mu$ being $\lambda$ we prove that either ${\tt T} \hookrightarrow^{\lambda} \circ \hookrightarrow^{\4} {\tt S}$ or ${\tt T} \hookrightarrow^{\lambda} \circ \hookrightarrow^{\lambda} \circ \hookrightarrow^{\4} {\tt S}$ (see Figure \ref{fig:4LambdaNormalization}). Finally, for $\mu$ being $\J$ we show that either ${\tt T} \hookrightarrow^{J} \circ \hookrightarrow^{\4} {\tt S}$ or ${\tt T} \hookrightarrow^{J} \circ \hookrightarrow^{J} \circ \hookrightarrow^{\4} {\tt S}$ (see Figure \ref{fig:4JNormalization}).
\end{proof}

\begin{figure}[t]
    \centering
\begin{subfigure}{1\textwidth}
\begin{subfigure}{.4\textwidth}
\centering
\begin{subfigure}{.2\textwidth}
    \centering
    \begin{tikzpicture}   
     \node[state] (b) at (0,0) {};
     \node[state][label=right: \hspace{0.1cm}$\hookrightarrow^{\4}$] (d) [below =of b] {};
     \node[state] (f) [below =of d] {};
     \draw[-latex] (b) -- (d) node[fill=white,inner sep=2pt,midway] {$\alpha$}; 
     \draw[-latex] (d) -- (f) node[fill=white,inner sep=2pt,midway] {$\alpha$};
    \end{tikzpicture} 
\end{subfigure}
\hspace{0.2cm}
\begin{subfigure}{.2\textwidth}
    \centering
    \begin{tikzpicture}   
    \node[state][label=right: \hspace{0.2cm}$\hookrightarrow^{\lambda}$] (b) at (0,0) {};
    \node[state] (d) [below =of b] {};
    \draw[-latex] (b) -- (d) node[fill=white,inner sep=2pt,midway] {$\alpha$}; 
    \end{tikzpicture} 
\end{subfigure}
\hspace{0.1cm}
\begin{subfigure}{.2\textwidth}
    \centering
    \begin{tikzpicture}   
    \node[state] (b) at (0,0) {};
    \node[state] (d) [below =of b] {};
    \draw[-latex] (b) -- (d) node[fill=white,inner sep=2pt,midway] {$\beta$}; 
    \end{tikzpicture} 
\end{subfigure}
\caption{${\tt T} \hookrightarrow^{\4} \circ \hookrightarrow^{\lambda} {\tt S}$.}
\end{subfigure}
\begin{subfigure}{0.6\textwidth}
\centering
\begin{subfigure}{.15\textwidth}
    \centering
    \begin{tikzpicture}   
     \node[state] (b) at (0,0) {};
     \node[state][label=right: \hspace{0.1cm}$\hookrightarrow^{\lambda}$] (d) [below =of b] {};
     \node[state] (f) [below =of d] {};
     \draw[-latex] (b) -- (d) node[fill=white,inner sep=2pt,midway] {$\alpha$}; 
     \draw[-latex] (d) -- (f) node[fill=white,inner sep=2pt,midway] {$\alpha$};
    \end{tikzpicture} 
\end{subfigure}
\hspace{0.1cm}
\begin{subfigure}{.15\textwidth}
    \centering
    \begin{tikzpicture}   
     \node[state] (b) at (0,0) {};
     \node[state][label=right: \hspace{0.1cm}$\hookrightarrow^{\lambda}$] (d) [below =of b] {};
     \node[state] (f) [below =of d] {};
     \draw[-latex] (b) -- (d) node[fill=white,inner sep=2pt,midway] {$\beta$}; 
     \draw[-latex] (d) -- (f) node[fill=white,inner sep=2pt,midway] {$\alpha$};
    \end{tikzpicture} 
\end{subfigure}
\hspace{0.1cm}
\begin{subfigure}{.15\textwidth}
    \centering
    \begin{tikzpicture}   
     \node[state] (b) at (0,0) {};
     \node[state][label=right: \hspace{0.2cm}$\hookrightarrow^{\4}$] (d) [below =of b] {};
     \node[state] (f) [below =of d] {};  
     \draw[-latex] (b) -- (d) node[fill=white,inner sep=2pt,midway] {$\beta$}; 
     \draw[-latex] (d) -- (f) node[fill=white,inner sep=2pt,midway] {$\beta$};
    \end{tikzpicture} 
\end{subfigure}
\hspace{0.1cm}
\begin{subfigure}{.08\textwidth}
    \centering
    \begin{tikzpicture}   
    \node[state] (b) at (0,0) {};
    \node[state] (d) [below =of b] {};
    \draw[-latex] (b) -- (d) node[fill=white,inner sep=2pt,midway] {$\beta$}; 
    \end{tikzpicture} 
\end{subfigure}
\caption{${\tt T} \hookrightarrow^{\lambda} \circ \hookrightarrow^{\lambda} \circ \hookrightarrow^{\4} {\tt S}$.}
\end{subfigure}
\end{subfigure}
\caption{${\tt T} \hookrightarrow^{\4} \circ \hookrightarrow^{\lambda} {\tt S} \Longrightarrow {\tt T} \hookrightarrow^{\lambda} \circ \hookrightarrow^{\lambda} \circ \hookrightarrow^{\4} {\tt S}$.}
\label{fig:4LambdaNormalization}
\end{figure}

\begin{corollary}
\label{CorollaryDecreasing}
Let $\delta$ and $\theta$ be decreasing and modal rewriting rules respectively, and $\Omega_\delta$ be a list of decreasing rewriting rules.
\begin{enumerate}
    \item If ${\tt T} \hookrightarrow^{\Omega_\delta} \circ \hookrightarrow^{\pi^+} {\tt S}$, then ${\tt T} \hookrightarrow^{\pi^+} \circ \hookrightarrow^{\Omega'_\delta} {\tt S}$ for $\Omega'_\delta$ a list of decreasing rules.
    \item If ${\tt T} \hookrightarrow^{\delta} \circ \hookrightarrow^{\theta *} {\tt S}$, then ${\tt T} \hookrightarrow^{\theta *} \circ \hookrightarrow^{\delta} {\tt S}$.
    \item If ${\tt T} \hookrightarrow^{\Omega_\delta} \circ \hookrightarrow^{\mu *} {\tt S}$ for $\mu \in \{\pi^+, \lambda, \J\}$, then ${\tt T} \hookrightarrow^{\mu *} \circ \hookrightarrow^{\Omega'_\delta} {\tt S}$ for $\Omega'_\delta$ a list of decreasing rewriting rules.
\end{enumerate}
\end{corollary}

\begin{proof}
The first statement follows by induction on the length of $\Omega_\delta$ and Proposition \ref{DecreasingWithModal}. The second one is shown by induction on the number of modal rewriting rules applied and Proposition \ref{DecreasingWithModal}. The third result follows by cases on $\mu$. For $\mu \in \{\lambda, \J\}$, we conclude by induction on the length of $\Omega_\delta$ and the second statement. Otherwise, we conclude by induction on the number of $\mu$-rules applied and the first statement.
\end{proof}

\begin{proposition}
\label{ChildDupModal}
Let $\theta$ be a modal rewriting rule. If ${\tt T} \hookrightarrow^{\theta} \circ \hookrightarrow^{\pi^+} {\tt S}$, then ${\tt T} \hookrightarrow^{\pi^+ *} \circ \hookrightarrow^{\theta *} \circ \hookrightarrow^{\sigma *} {\tt S}$.
\end{proposition}

\begin{proof}
By induction on the length of the position of $\theta$ and by cases on the length of the position of $\pi^+$. If $\theta$ is $\lambda$, we show that either ${\tt T} \hookrightarrow^{\pi^+} \circ \hookrightarrow^{\lambda} {\tt S}$ or ${\tt T} \hookrightarrow^{\pi^+} \circ \hookrightarrow^{\lambda} \circ \hookrightarrow^{\lambda}{\tt S}$ holds (see Figure \ref{fig:Pi+Normalization}). For $\theta$ being $\J$, we prove that either ${\tt T} \hookrightarrow^{ \pi^+} \circ \hookrightarrow^{\J} {\tt S}$ or ${\tt T} \hookrightarrow^{\pi^+} \circ \hookrightarrow^{\J} \circ \hookrightarrow^{\J} {\tt S}$ (see Figure \ref{fig:Pi+Normalization}) or ${\tt T} \hookrightarrow^{ \pi^+} \circ \hookrightarrow^{ \pi^+} \circ \hookrightarrow^{\J} \circ \hookrightarrow^{\J} {\tt S}$ (see Figure \ref{fig:JPi+1Normalization}) or ${\tt T} \hookrightarrow^{\pi^+} \circ \hookrightarrow^{\J} \circ \hookrightarrow^{\J} \circ \hookrightarrow^{\sigma} {\tt S}$ (see Figure \ref{fig:JPi+2Normalization}).
\end{proof}

We now want to show permutation of the application of multiple $\pi^+$-rules after multiple modal rewriting rules. To show permutability of $\pi^+$-rules following a $\J$-rule, we need to reorganize duplications applied at positions in increasing depth in the tree. Furthermore, we require flexibility to choose which branch to duplicate first.

For readability, we write ${\tt T} \hookrightarrow^{\pi^+(\mathbf{k},i)} {\tt S}$ to denote that ${\tt S}$ is obtained by applying the $\pi^+$-rule to duplicate the $i$th child at position $\mathbf{k} \in {\sf Pos}({\tt T})$. Similarly, ${\tt T} \hookrightarrow^{\sigma(\mathbf{k},i,j)} {\tt S}$ denotes the application of the $\sigma$-rule to permute the $i$th and $j$th children at position $\mathbf{k} \in {\sf Pos}({\tt T})$. Given this notation, the following result aids in reorganizing duplications at positions in increasing depth.

\begin{lemma}[Depth-level permutability of $\pi^+$]
\label{PermutDepth}
\leavevmode\vspace{0.01em}
\begin{enumerate}
    \item If ${\tt T} \hookrightarrow^{\pi^+(l\mathbf{k},i)} \circ \hookrightarrow^{\pi^+(\epsilon,j)} {\tt S}$ for $j \neq l$, then ${\tt T} \hookrightarrow^{\pi^+(\epsilon,j)} \circ \hookrightarrow^{\pi^+((l+1)\mathbf{k},i)} {\tt S}$.
    \item If ${\tt T} \hookrightarrow^{\pi^+(j\mathbf{k},i)} \circ \hookrightarrow^{\pi^+(\epsilon,j)} {\tt S}$, then ${\tt T} \hookrightarrow^{\pi^+(\epsilon,j)} \circ \hookrightarrow^{\pi^+((j+1)\mathbf{k},i)} \circ \hookrightarrow^{\pi^+(1\mathbf{k},i)} {\tt S}$.
    \item If ${\tt T} \hookrightarrow^{\pi^+(l\mathbf{k}_l,i)} \circ \hookrightarrow^{\pi^+ (n\mathbf{k}_n,j)} {\tt S}$ for $|n\mathbf{k}_n| < |l\mathbf{k}_l|$, then ${\tt T} \hookrightarrow^{\pi^+(n\mathbf{k}_n,j)} \circ \hookrightarrow^{\pi^+ (l\mathbf{k}_l,i)} {\tt S}$.
\end{enumerate}
\end{lemma}

\begin{proof}
    The first two statements are trivially satisfied by definition. The third follows by induction on the tree structure of ${\tt T}$.
\end{proof}

\begin{figure}[t]
\centering
\begin{subfigure}{1\textwidth}
\centering
\begin{subfigure}{.2\textwidth}
    \centering
    \begin{tikzpicture}  
        \node[state] (b) at (0,0) {};
        \node[state] (c) [below left =of b] {};
        \node[state][label=right: \hspace{0.3cm}$\hookrightarrow^{\4}$] (d) [below right =of b] {};
        \node[state] (f) [below =of c] {};
        \draw[-latex] (b) -- (c) node[fill=white,inner sep=2pt,midway] {{\small $\alpha$}};
        \draw[-latex] (c) -- (f) node[fill=white,inner sep=2pt,midway] {{\small $\alpha$}};
        \draw[-latex] (b) -- (d) node[fill=white,inner sep=2pt,midway] {{\small $\beta$}};
    \end{tikzpicture}
\end{subfigure}
\begin{subfigure}{.2\textwidth}
    \centering
    \begin{tikzpicture}  
        \node[state] (b) at (0,0) {};
        \node[state][label=right: \hspace{0.3cm}$\hookrightarrow^{\J}$] (d) [below right =of b] {};
        \node[state] (f) [below left =of b] {};
        \node[emptystate] (h) [below left=of f] {};
        \draw[-latex] (b) -- (f) node[fill=white,inner sep=2pt,midway] {{\small $\alpha$}};
        \draw[-latex] (b) -- (d) node[fill=white,inner sep=2pt,midway] {{\small $\beta$}}; 
    \end{tikzpicture} 
\end{subfigure}
\hspace{0.6cm}
\begin{subfigure}{.22\textwidth}
    \centering
    \begin{tikzpicture}   
     \node[state] (b) at (0,0) {};
     \node[state] (d) [below =of b] {};
     \node[state] (f) [below =of d] {};  
     \draw[-latex] (b) -- (d) node[fill=white,inner sep=2pt,midway] {$\alpha$}; 
     \draw[-latex] (d) -- (f) node[fill=white,inner sep=2pt,midway] {$\beta$};
    \end{tikzpicture} 
\end{subfigure}
\caption{${\tt T} \hookrightarrow^{\4} \circ \hookrightarrow^{\J} {\tt S}$.}
\end{subfigure}
\vspace{0.5cm}
\begin{subfigure}{1\textwidth}
\centering
\begin{subfigure}{.2\textwidth}
    \centering
    \begin{tikzpicture}  
        \node[state] (b) at (0,0) {};
        \node[state] (c) [below left =of b] {};
        \node[state][label=right: \hspace{0.2cm}$\hookrightarrow^{\J}$] (d) [below right =of b] {};
        \node[state] (f) [below =of c] {};
        \draw[-latex] (b) -- (c) node[fill=white,inner sep=2pt,midway] {{\small $\alpha$}};
        \draw[-latex] (c) -- (f) node[fill=white,inner sep=2pt,midway] {{\small $\alpha$}};
        \draw[-latex] (b) -- (d) node[fill=white,inner sep=2pt,midway] {{\small $\beta$}};
    \end{tikzpicture} 
\end{subfigure}
\hspace{0.1cm}
\begin{subfigure}{.2\textwidth}
    \centering
    \begin{tikzpicture}   
        \node[state] (b) at (0,0) {};
        \node[state][label=right: \hspace{0.5cm}$\hookrightarrow^{\J}$] (c) [below =of b] {};
        \node[state] (d) [below left =of c] {};
        \node[state] (f) [below right =of c] {};
        \draw[-latex] (b) -- (c) node[fill=white,inner sep=2pt,midway] {{\small $\alpha$}};
        \draw[-latex] (c) -- (f) node[fill=white,inner sep=2pt,midway] {{\small $\beta$}};
        \draw[-latex] (c) -- (d) node[fill=white,inner sep=2pt,midway] {{\small $\alpha$}};
    \end{tikzpicture} 
\end{subfigure}
\hspace{0.2cm}
\begin{subfigure}{.1\textwidth}
    \centering
    \begin{tikzpicture}   
         \node[state] (b) at (0,0) {};
        \node[state]  (c) [below =of b] {};
        \node[state][label=right: \hspace{0.3cm}$\hookrightarrow^{\4}$] (f) [below =of c] {};
        \node[state] (d) [below =of f] {};
        \draw[-latex] (b) -- (c) node[fill=white,inner sep=2pt,midway] {{\small $\alpha$}};
        \draw[-latex] (c) -- (f) node[fill=white,inner sep=2pt,midway] {{\small $\alpha$}};
        \draw[-latex] (f) -- (d) node[fill=white,inner sep=2pt,midway] {{\small $\beta$}};
    \end{tikzpicture} 
\end{subfigure}
\begin{subfigure}{.07\textwidth}
    \centering
    \begin{tikzpicture}   
    \node[state] (b) at (0,0) {};
     \node[state] (d) [below =of b] {};
     \node[state] (f) [below =of d] {};  
     \draw[-latex] (b) -- (d) node[fill=white,inner sep=2pt,midway] {$\alpha$}; 
     \draw[-latex] (d) -- (f) node[fill=white,inner sep=2pt,midway] {$\beta$};
    \end{tikzpicture} 
\end{subfigure}
\caption{${\tt T} \hookrightarrow^{\J} \circ \hookrightarrow^{\J} \circ \hookrightarrow^{\4} {\tt S}$.}
\end{subfigure}
\caption{${\tt T} \hookrightarrow^{\4} \circ \hookrightarrow^{\J} {\tt S} \Longrightarrow {\tt T} \hookrightarrow^{\J} \circ \hookrightarrow^{\J} \circ \hookrightarrow^{\4} {\tt S}$.}
\label{fig:4JNormalization}
\end{figure}

Likewise, the lemma below permits to choose which branch to duplicate first. 

\begin{remark}[Width-level permutability of $\pi^+$]
\label{PermutFixedDepth}
\leavevmode\vspace{0.01em}
\begin{enumerate}
    \item If ${\tt T} \hookrightarrow^{\pi^+(\epsilon,i)} \circ \hookrightarrow^{\pi^+(\epsilon,j)} {\tt S}$ for $1 < j$ and $j \neq i+1$, then ${\tt T} \hookrightarrow^{\pi^+ (\epsilon,j-1)} \circ \hookrightarrow^{\pi^+ (\epsilon,i+1)} \hookrightarrow^{\sigma (\epsilon,1,2)} {\tt S}$.
    \item If ${\tt T} \hookrightarrow^{\pi^+(n\mathbf{k}_n,i)} \circ \hookrightarrow^{\pi^+(l\mathbf{k}_l,j)} {\tt S}$ for $n \neq l$, then ${\tt T} \hookrightarrow^{\pi^+(l\mathbf{k}_l,j)} \circ \hookrightarrow^{\pi^+(n\mathbf{k}_n,i)} {\tt S}$.
\end{enumerate}
\end{remark}

Therefore, we can show permutation of the application of arbitrary $\pi^+$-rules after a $\J$-rule.

\begin{figure}[t]
    \centering
\begin{subfigure}{1\textwidth}
\centering
\begin{subfigure}{.3\textwidth}
    \centering
    \begin{tikzpicture}        
        \node[state][label=right: \hspace{0.8cm}$\hookrightarrow^{\J}$]  (b) at (0,0) {};
        \node[state] (c) [below left =of b] {};
        \node[state](d) [below right =of b] {};
        \draw[-latex] (b) -- (c) node[fill=white,inner sep=2pt,midway] {{\small $\alpha$}};
        \draw[-latex] (b) -- (d) node[fill=white,inner sep=2pt,midway] {{\small $\beta$}}; 
    \end{tikzpicture} 
\end{subfigure}
\begin{subfigure}{.1\textwidth}
    \centering
    \begin{tikzpicture}
    \node[state] (b) at (0,0) {};
     \node[state][label=right: \hspace{0.3cm}$\hookrightarrow^{\pi^+}$] (d) [below =of b] {};
     \node[state] (f) [below =of d] {}; 
     \draw[-latex] (b) -- (d) node[fill=white,inner sep=2pt,midway] {$\alpha$}; 
     \draw[-latex] (d) -- (f) node[fill=white,inner sep=2pt,midway] {$\beta$};
    \end{tikzpicture} 
\end{subfigure}
\hspace{0.2cm}
\begin{subfigure}{.2\textwidth}
    \centering
    \begin{tikzpicture}        
    \node[state] (b) at (0,0) {};
     \node[state] (d) [below left =of b] {};
     \node[state] (x) [below right =of b] {};
     \node[state] (f) [below =of d] {}; 
     \node[state] (z) [below =of x] {};
     \draw[-latex] (b) -- (d) node[fill=white,inner sep=2pt,midway] {$\alpha$}; 
     \draw[-latex] (d) -- (f) node[fill=white,inner sep=2pt,midway] {$\beta$};
     \draw[-latex] (b) -- (x) node[fill=white,inner sep=2pt,midway] {$\alpha$}; 
     \draw[-latex] (x) -- (z) node[fill=white,inner sep=2pt,midway] {$\beta$};
    \end{tikzpicture} 
\end{subfigure}
\caption{${\tt T} \hookrightarrow^{\J} \circ \hookrightarrow^{\pi^+} {\tt S}$.}
\end{subfigure}
\begin{subfigure}{1\textwidth}
\centering
\begin{subfigure}{.2\textwidth}
    \centering
    \begin{tikzpicture}        
        \node[state][label=right: \hspace{0.8cm}$\hookrightarrow^{\pi^+} \circ \hookrightarrow^{\pi^+}$] (b) at (0,0) {};
        \node[state] (c) [below left =of b] {};
        \node[state] (d) [below right =of b] {};
        \draw[-latex] (b) -- (c) node[fill=white,inner sep=2pt,midway] {{\small $\alpha$}};
        \draw[-latex] (b) -- (d) node[fill=white,inner sep=2pt,midway] {{\small $\beta$}}; 
    \end{tikzpicture} 
\end{subfigure}
\hspace{0.2cm}
\begin{subfigure}{.2\textwidth}
    \centering
    \begin{tikzpicture}        
        \node[state][label=right: \hspace{0.3cm}$\hookrightarrow^{\J}$] (b) at (0,0) {};
        \node[state] (c) [below =of b] {};
        \node[state] (d) [below right =of b] {};
        \node[state] (x) [below left =of b] {};
        \node[state] (y) [left =of x] {};
        \draw[-latex] (b) -- (c) node[fill=white,inner sep=2pt,midway] {{\small $\alpha$}};
        \draw[-latex] (b) -- (y) node[fill=white,inner sep=2pt,midway] {{\small $\alpha$}};
        \draw[-latex] (b) -- (d) node[fill=white,inner sep=2pt,midway] {{\small $\beta$}}; 
        \draw[-latex] (b) -- (x) node[fill=white,inner sep=2pt,midway] {{\small $\beta$}};
    \end{tikzpicture} 
\end{subfigure}
\hspace{1cm}
\begin{subfigure}{.2\textwidth}
    \centering
    \begin{tikzpicture}        
    \node[state] (b) at (0,0) {};
     \node[state] (d) [below left =of b] {};
     \node[state] (x) [below =of b] {};
     \node[state] (f) [below =of d] {}; 
     \node[state][label=right: \hspace{0.15cm}$\hookrightarrow^{\J}$] (z) [below right =of b] {}; 
     \draw[-latex] (b) -- (d) node[fill=white,inner sep=2pt,midway] {$\alpha$}; 
     \draw[-latex] (d) -- (f) node[fill=white,inner sep=2pt,midway] {$\beta$};
     \draw[-latex] (b) -- (x) node[fill=white,inner sep=2pt,midway] {$\alpha$}; 
     \draw[-latex] (b) -- (z) node[fill=white,inner sep=2pt,midway] {$\beta$};
    \end{tikzpicture} 
\end{subfigure}
\hspace{0.5cm}
\begin{subfigure}{.2\textwidth}
    \centering
    \begin{tikzpicture}        
    \node[state] (b) at (0,0) {};
     \node[state] (d) [below left =of b] {};
     \node[state] (x) [below right =of b] {};
     \node[state] (f) [below =of d] {}; 
     \node[state] (z) [below =of x] {}; 
     \draw[-latex] (b) -- (d) node[fill=white,inner sep=2pt,midway] {$\alpha$}; 
     \draw[-latex] (d) -- (f) node[fill=white,inner sep=2pt,midway] {$\beta$};
     \draw[-latex] (b) -- (x) node[fill=white,inner sep=2pt,midway] {$\alpha$}; 
     \draw[-latex] (x) -- (z) node[fill=white,inner sep=2pt,midway] {$\beta$};
    \end{tikzpicture} 
\end{subfigure}
\caption{${\tt T} \hookrightarrow^{\pi^+} \circ \hookrightarrow^{\pi^+} \circ \hookrightarrow^{\J} \circ \hookrightarrow^{\J} {\tt S}$.}
\end{subfigure}
\caption{${\tt T} \hookrightarrow^{\J} \circ \hookrightarrow^{\pi^+} {\tt S} \Longrightarrow {\tt T} \hookrightarrow^{\pi^+} \circ \hookrightarrow^{\pi^+} \circ\hookrightarrow^{\J} \circ \hookrightarrow^{\J} {\tt S}$.}
\label{fig:JPi+1Normalization}
\end{figure}

\begin{proposition}
\label{JPi+}
If ${\tt T} \hookrightarrow^{\J} \circ \hookrightarrow^{\pi^+ *} {\tt S}$, then ${\tt T} \hookrightarrow^{\pi^+ *} \circ \hookrightarrow^{\J *} \circ \hookrightarrow^{\sigma *} {\tt S}$.
\end{proposition}

\begin{proof}
If the $\J$-rule and the $\pi^+$-rules are applied in different positions of the tree such that the performed transformations are disjoint, their permutation is proven straightforward as shown in Proposition \ref{ChildDupModal}. Therefore, we will use Lemma \ref{PermutDepth} and Remark \ref{PermutFixedDepth} to reorganise the application of the $\pi^+$-rules to easily permute the disjoint cases. Otherwise, in cases where the subtree \textit{affected} by the $\J$-rule is duplicated, we use a unique labeling of the nodes to track the number of times those subtrees are duplicated. Specifically, we label the nodes with their positions. For ${\tt T} \hookrightarrow^{\J} \hookrightarrow^{*} {\tt S}$ such that $\J$ has been performed at position $n\mathbf{k} \in {\sf Pos}({\tt T})$, we call the \textit{$\J$-label} to the label $n$. Moreover, if $\J$ has been performed at the empty position relating the $i$th and the $j$th children according to the notation of its definition, we call the \textit{upper $\J$-label} to the label $i$ and the \textit{lower $\J$-label} to the label $j$.

Therefore, the proof proceeds by induction on the length of the position in which the $\J$-rule has been applied.
For the $\J$-rule applied at the empty position, we proceed by induction on the occurrences of the upper $\J$-label. The inductive step is shown by Proposition \ref{ChildDupModal} and Corollary \ref{CorollarySigma}, using Lemma \ref{PermutDepth} and Remark \ref{PermutFixedDepth} to get a suitable reorganisation of the applications of the $\pi^+$-rules. The base case in which the upper $\J$-label occurs once in ${\tt S}$ is shown by induction on the occurrences of the lower $\J$-label using the same strategy.

Lastly, if the $\J$-rule is applied at a non-empty position, we proceed by induction on the occurrences of the $\J$-label. The base case follows by Proposition \ref{ChildDupModal} using Lemma \ref{PermutDepth} and Remark \ref{PermutFixedDepth} together with the inductive hypothesis on the position in which $\J$ has been performed. The inductive step is similarly proven by Proposition \ref{ChildDupModal} and Corollary \ref{CorollarySigma} using Lemma \ref{PermutDepth} and Remark \ref{PermutFixedDepth}.
\end{proof}

\begin{corollary}
\label{CorollaryExtensional}
Let $\Omega_\diamond$ be a list of modal rewriting rules.
\begin{enumerate}
    \item If ${\tt T} \hookrightarrow^{\lambda} \circ \hookrightarrow^{\pi^+ *} {\tt S}$, then ${\tt T} \hookrightarrow^{\pi^+ *} \circ \hookrightarrow^{\lambda *} {\tt S}$ by applying the same number of $\pi^+$-rules.
    \item If ${\tt T} \hookrightarrow^{\Omega_\diamond} \circ \hookrightarrow^{\pi^+ *} {\tt S}$, then ${\tt T} \hookrightarrow^{\pi^+ *} \circ \hookrightarrow^{\Omega'_\diamond} \circ \hookrightarrow^{\sigma *} {\tt S}$ for $\Omega'_\diamond$ a list of modal rewriting rules.
\end{enumerate}
\end{corollary}

\begin{proof}
The first statement follows by induction on the number of $\pi^+$-rules applied and Proposition \ref{ChildDupModal}. The second statement follows by induction on the length of $\Omega_\diamond$. The inductive step is shown by cases on the considered modal rewriting rule: the $\lambda$-rule case follows by the first statement and the $\J$-rule case is shown by Proposition \ref{JPi+} and Corollary \ref{PermutationWithRest}.
\end{proof}

\begin{figure}
    \centering
\begin{subfigure}{1\textwidth}
\centering
\begin{subfigure}{.2\textwidth}
    \centering
    \begin{tikzpicture}  
        \node[state] (b) at (0,0) {};
        \node[state] (c) [below left =of b] {};
        \node[state][label=right: \hspace{0.3cm}$\hookrightarrow^{\J}$] (d) [below right =of b] {};
        \node[state] (e) [below =of c] {};
        \draw[-latex] (b) -- (c) node[fill=white,inner sep=2pt,midway] {{\small $\alpha$}};
        \draw[-latex] (b) -- (d) node[fill=white,inner sep=2pt,midway] {{\small $\beta$}};
        \draw[-latex] (c) -- (e) node[fill=white,inner sep=2pt,midway] {{\small $\gamma$}}; 
    \end{tikzpicture}
\end{subfigure}
\begin{subfigure}{.2\textwidth}
    \centering
    \begin{tikzpicture}   
     \node[state] (b) at (0,0) {};
     \node[state][label=right: \hspace{0.8cm}$\hookrightarrow^{\pi^+}$] (d) [below =of b] {};
     \node[state] (f) [below left =of d] {}; 
     \node[state] (g) [below right =of d] {};
     \draw[-latex] (b) -- (d) node[fill=white,inner sep=2pt,midway] {$\alpha$}; 
     \draw[-latex] (d) -- (f) node[fill=white,inner sep=2pt,midway] {$\gamma$}; 
     \draw[-latex] (d) -- (g) node[fill=white,inner sep=2pt,midway] {$\beta$};
    \end{tikzpicture} 
\end{subfigure}
\begin{subfigure}{.2\textwidth}
    \centering
    \begin{tikzpicture}   
     \node[state] (b) at (0,0) {};
     \node[state] (d) [below =of b] {};
     \node[state] (f) [below =of d] {}; 
     \node[state] (g) [below right =of d] {};
     \node[state] (h) [below left =of d] {};
     \draw[-latex] (b) -- (d) node[fill=white,inner sep=2pt,midway] {$\alpha$}; 
     \draw[-latex] (d) -- (f) node[fill=white,inner sep=2pt,midway] {$\gamma$}; 
     \draw[-latex] (d) -- (g) node[fill=white,inner sep=2pt,midway] {$\beta$};
     \draw[-latex] (d) -- (h) node[fill=white,inner sep=2pt,midway] {$\beta$};
    \end{tikzpicture} 
\end{subfigure}
\caption{${\tt T} \hookrightarrow^{\J} \circ \hookrightarrow^{\pi^+} {\tt S}$.}
\end{subfigure}
\vspace{0.5cm}
\begin{subfigure}{1.\textwidth}
\centering
\begin{subfigure}{.2\textwidth}
    \begin{tikzpicture}  
        \node[state] (b) at (0,0) {};
        \node[state] (c) [below left =of b] {};
        \node[state][label=right: \hspace{0.2cm}$\hookrightarrow^{\pi^+}$] (d) [below right =of b] {};
        \node[state] (e) [below =of c] {};
        \draw[-latex] (b) -- (c) node[fill=white,inner sep=2pt,midway] {{\small $\alpha$}};
        \draw[-latex] (b) -- (d) node[fill=white,inner sep=2pt,midway] {{\small $\beta$}};
        \draw[-latex] (c) -- (e) node[fill=white,inner sep=2pt,midway] {{\small $\gamma$}};
    \end{tikzpicture} 
\end{subfigure}
\hspace{0.8cm}
\begin{subfigure}{.2\textwidth}
    \centering
    \begin{tikzpicture}   
        \node[state] (b) at (0,0) {};
        \node[state] (c) [below =of b] {};
        \node[state][label=right: \hspace{0.2cm}$\hookrightarrow^{\J} \circ \hookrightarrow^{\J} $] (d) [below right =of b] {};
        \node[state] (e) [below left =of b] {};
        \node[state] (f) [below =of c] {};
        \draw[-latex] (b) -- (c) node[fill=white,inner sep=2pt,midway] {{\small $\alpha$}};
        \draw[-latex] (b) -- (d) node[fill=white,inner sep=2pt,midway] {{\small $\beta$}};
        \draw[-latex] (b) -- (e) node[fill=white,inner sep=2pt,midway] {{\small $\beta$}};
        \draw[-latex] (c) -- (f) node[fill=white,inner sep=2pt,midway] {{\small $\gamma$}};
    \end{tikzpicture} 
\end{subfigure}
\hspace{0.6cm}
\begin{subfigure}{.1\textwidth}
    \centering
    \begin{tikzpicture}   
    \node[state] (b) at (0,0) {};
     \node[state][label=right: \hspace{0.8cm}$\hookrightarrow^{\sigma}$] (d) [below =of b] {};
     \node[state] (f) [below =of d] {}; 
     \node[state] (g) [below right =of d] {};
     \node[state] (h) [below left =of d] {};
     \draw[-latex] (b) -- (d) node[fill=white,inner sep=2pt,midway] {$\alpha$}; 
     \draw[-latex] (d) -- (f) node[fill=white,inner sep=2pt,midway] {$\beta$}; 
     \draw[-latex] (d) -- (g) node[fill=white,inner sep=2pt,midway] {$\beta$};
     \draw[-latex] (d) -- (h) node[fill=white,inner sep=2pt,midway] {$\gamma$};
    \end{tikzpicture} 
\end{subfigure}
\hspace{1.5cm}
\begin{subfigure}{.1\textwidth}
    \centering
    \begin{tikzpicture}   
     \node[state] (b) at (0,0) {};
     \node[state] (d) [below =of b] {};
     \node[state] (f) [below =of d] {}; 
     \node[state] (g) [below right =of d] {};
     \node[state] (h) [below left =of d] {};
     \draw[-latex] (b) -- (d) node[fill=white,inner sep=2pt,midway] {$\alpha$}; 
     \draw[-latex] (d) -- (f) node[fill=white,inner sep=2pt,midway] {$\gamma$}; 
     \draw[-latex] (d) -- (g) node[fill=white,inner sep=2pt,midway] {$\beta$};
     \draw[-latex] (d) -- (h) node[fill=white,inner sep=2pt,midway] {$\beta$};
    \end{tikzpicture} 
\end{subfigure}
\caption{${\tt T} \hookrightarrow^{\pi^+} \circ \hookrightarrow^{\J} \circ \hookrightarrow^{\J} \circ \hookrightarrow^{\sigma} {\tt S}$.}
\end{subfigure}
\caption{${\tt T} \hookrightarrow^{\J} \circ \hookrightarrow^{\pi^+} {\tt S} \Longrightarrow {\tt T} \hookrightarrow^{\pi^+} \circ \hookrightarrow^{\J} \circ \hookrightarrow^{\J} \circ \hookrightarrow^{\sigma} {\tt S}$.}
\label{fig:JPi+2Normalization}
\end{figure}

Finally, we present the main theorem of the section.

\begin{theorem}[Rewrite Normalization Theorem]
\label{Normalization}
If ${\tt T} \hookrightarrow^{*} {\tt S}$, then ${\tt T} \hookrightarrow^\Omega {\tt S}$ for $\Omega$ a normal rewriting sequence. 
\end{theorem}

\begin{proof}
By induction on the number of rewriting steps that are performed using Corollaries \ref{CorollarySigma}, \ref{CorollaryAtomic}, \ref{CorollaryDecreasing} and \ref{CorollaryExtensional}.
\end{proof}

\section{Conclusions and future work}

We have provided a method for designing tree rewriting systems for different strictly positive logics and in particular re-cast the reflection calculus in this framework.
Although not explicitly stated, it easily follows from the presented techniques that a rewriting system for ${\bf K}^+$ can be defined by only considering atomic, structural, extensional and $\pi$-rules, and we are confident that our framework will find applications in positive fragments for other (poly)modal logics.
The use of abstract rewriting systems aids in the analysis of further properties like subformula property and admissibility of rules and provides a foundation for computationally efficient implementation.

Our work is based on a type-theoretic presentation of trees, which has the dual benefit of allowing for rules to be described in a precise and succinct manner and being particularly amenable to formalization in proof assistants.
Many of our results have already been implemented in the proof assistant Coq and our goal is to fully formalize our work.
Since rewriting normalization theorem and proof-theoretic investigations in general require checking multiple distinct cases in detail, the benefit of formalization is particularly clear in this type of proof, and indeed the community has been gravitating towards formalized proofs (see e.g.~\cite{ShillitoG22}).
In the framework of $\bf RC$, this has the added advantage that results often need not only to be {\em true,} but provable in suitable systems of arithmetic, and the latter can be made no more transparent than via formalization.
Last but not least, proofs implemented in Coq can be automatically extracted into fully verified algorithms, paving the road to reasoners with the highest degree of reliability attainable by current technology.

\bibliographystyle{plain}
\bibliography{bibliography}


\end{document}